\documentclass[12pt]{amsart}

\usepackage{amssymb,latexsym,amsthm, amsmath, comment, fullpage}
\usepackage[all]{xy}
\usepackage{graphicx, comment, MnSymbol, wrapfig}

\ifx\pdfoutput\undefined
\usepackage{hyperref}
\else
\usepackage[pdftex,colorlinks=true,linkcolor=blue,urlcolor=blue]{hyperref}
\fi

\theoremstyle{plain}
\newtheorem{theorem}{Theorem}

\newtheorem{proposition}{Proposition}

\newtheorem{lemma}{Lemma}
\theoremstyle{definition}
\newtheorem{definition}{Definition}
\newtheorem{remark}{Remark}

\newtheorem*{ack}{Acknowledgements}

\begin{document}

\title[Flat surfaces, Bratteli diagrams, and unique ergodicity \`{a} la Masur]
      {Flat surfaces, Bratteli diagrams, and unique ergodicity \`{a} la Masur}
\author{Rodrigo Trevi\~no}
\address{Brooklyn College, City University of New York}
\email{rodrigo@trevino.cat}
\date{\today}

\begin{abstract}
Recalling the construction of a flat surface from a Bratteli diagram, this paper considers the dynamics of the shift map on the space of all bi-infinite Bratteli diagrams as the renormalizing dynamics on a moduli space of flat surfaces of finite area. A criterion of unique ergodicity similar to that of Masur's for flat surface holds: if there is a subsequence of the renormalizing dynamical system which has a good accumulation point, the translation flow or Bratteli-Vershik transformation is uniquely ergodic. Related questions are explored.
\end{abstract}
\maketitle

\section{Introduction and statement of results}
A flat surface is a Riemann surface which has a flat metric in all but finitely points exceptional. When the genus is greater than 2, these exceptional points carry all the negative curvature which is necessary by the Gauss-Bonnet theorem. They are usually called singularities.

The types of dynamical systems studied on flat surface are called \textbf{translation flows}. When the genus of the surface is greater than 2, these systems have zero entropy but are not elliptic: there is a polynomial rate of divergence of trajectories, and thus the dynamics of translation flows on a higher-genus flat surface is not trivial, as is the case of translation flows on a flat torus. The study of the dynamical properties of these systems goes back many decades and is very rich since it has many deep connections to many other fields: polygonal billiards, study of the dynamics of interval exchange transformations, algebraic geometry, combinatorics, Bratteli diagrams, etc.

One of the main tools in the study of translation flows on compact flat surfaces is the \textbf{Teichm\"uller flow}. This flow is a hyperbolic flow on the moduli space of flat surfaces, the space of all confomally equivalent flat surfaces of the same topological type. It is a non-compact, finite-volume orbifold. The orbit of a surface under this flow is the one-parameter family of surfaces obtained from the original surface by a one-parameter family of deformations of the flat metric, thus getting a one-parameter family of flat surfaces. This family is the Teichm\"uller orbit of the surface in moduli space. 

In deforming a surface under the Teichm\"uller deformation, persistent geometric features sometimes have implications for the translation flow on the surface. The prime example of this is \textbf{Masur's criterion}, which states that if the translation flow on a flat surface is not uniquely ergodic, then the Teichm\"uller orbit of that surface is divergent, that is, it leaves every compact set of the moduli space. Another way of stating this is that if, as the surface is deformed under the Teichm\"uller deformation, there is a subsequence of surfaces which converge to a surface in moduli space, then the translation flow is uniquely ergodic. It is this idea of Masur which motivates this paper.

Another of the main motivations for this paper is the study of translation flows on surfaces of infinite genus. There has been significant developments in this field in the past 10 years. The main challenge in the study of flat surfaces of infinite genus has been the lack of moduli spaces. For compact surfaces, the moduli spaces of flat surfaces and different dynamical systems defined on them have served very well in helping understand relationships between the dynamics of translation flows and geometry of flat surfaces. Since no such spaces exist for surfaces of infinite genus, different authors have tackled this obstacle in different ways, for example by using Veech groups (e.g. \cite{Hinf,HW:erg, rodrigo:ergodicity, HT:evil}), spaces of flat structures (e.g. \cite{HooperImmersions1,HooperImmersions2}), covers (e.g. \cite{HLT:Ehrenfest, RT:ext, DHL:wind-tree,FU:nonerg}), to name a few.

I have been mostly interested in studying flat surfaces of infinite genus and finite area, and in this paper I only consider surfaces of finite area. The starting point for this paper is the joint work \cite{LT}, where a way of constructing flat surfaces from Bratteli diagrams was introduced, generalizing a point of view previously introduced by Bufetov \cite{bufetov:limitVershik, bufetov:limit}. Given a special type of infinite indexed graph, called a bi-infinite Bratteli diagram $\mathcal{B}$, along with some other structures called weight functions $w^+$ and $w^-$ (and collectively denoted by $w^\pm$), and orderings $\leq_{r,s}$ (all of this defined in \S \ref{sec:bratteli}), we can construct a unique flat surface $S(\mathcal{B},w^\pm,\leq_{r,s})$ whose dynamical and geometric properties are linked to the dynamical properties of $\mathcal{B}$. Starting from a bi-infinite Bratteli diagram $\mathcal{B}$ many surfaces can be constructed from it by choosing different weight functions $w^\pm$ as well as different orderings $\leq_{r,s}$.

By dynamical properties of $\mathcal{B}$ I mean the following: the bi-infinite Bratteli diagram $\mathcal{B}$ along with $\leq_{r,s}$ define two \textbf{Bratteli-Vershik transformations} on the space of infinite positive paths and infinite negative paths, respectively, of the Bratteli diagram $\mathcal{B}$ (this is explained in \S \ref{sec:bratteli}). The dynamics of these two transformations are intricately related to the dynamics of the vertical and horizontal translation flows on the flat surface constructed from $\mathcal{B}$ and $\leq_{r,s}$. In particular, the ergodicity of the Bratteli-Vershik transformations is equivalent to the ergodicity of the flows. One of the goals of \cite{LT} was to exploit this relationship in order to obtain results about the dynamics of Bratteli-Vershik transformations through the study of translation flows, and vice-versa. It should be mentioned that the surfaces obtained through this construction can have both finite and infinite genus. In fact, it is believed that a typical surface from this construction has infinite genus. For the reader who is unfamiliar with the dictionary developed in \cite{LT}, Appendix \ref{sec:fib} consists of an explicit example of how one starts with a Bratteli daigram, builds a surface from it, and how the renormalization operation works.

One of the key features of the relationship between Bratteli diagrams and flat surfaces constructed from them is the functoriality property proved in \cite[Proposition 6.3]{LT} and stated here as Proposition \ref{prop:shift} (a version of it has been relied on already by Bufetov \cite{bufetov:limit, bufetov:limitVershik}). This merely illustrates that the Teichm\"uller deformation of a flat surface constructed from a Bratteli diagram is manifested by the shifting indices on the Bratteli diagram and then constructing a flat surface from it. This is described in detail in \S \ref{subsec:renorm}.

In this paper I study the dynamics of the shift map $\sigma:\beth\rightarrow \beth$ on the space of all bi-infinite Bratteli diagrams $\beth$. Let $\mathfrak{M}$ be the set of all matrices with non-negative integer entries. Since it is countable, we can code this set and consider the set $\beth$ of all bi-infinite Bratteli diagrams as a subshift of the full shift on $\mathfrak{M}^\mathbb{Z}$. This is a countable state Markov shift of infinite topological entropy (it contains the full $n$-shift for all $n$). We topologize the set $\beth$ by cylinder sets.

By the functoriality property, we can think of $\beth$ as a ``moduli space'' of flat surfaces of finite area. Not all flat surfaces of finite area can be obtained through constructions from elements in $\beth$ (e.g. the open unit disk), but most of the ones with interesting dynamics can, and the ones with recurrent horizontal and vertical flows can. Thus the space $\beth$ serves as a good candidate for a space where the renormalization dynamics coming from the shift map can yield information about the translation flow of flat surfaces constructed from diagrams $\mathcal{B}\in\beth$. In particular, since we are treating $\beth$ as a moduli space, we would like to know whether something like Masur's criterion holds in this space. The purpose of this paper is to prove such a result and explore related ones.

\begin{theorem}
\label{thm:main}
Let $\mathcal{B}\in\beth$ be a Bratteli diagram. Suppose that the orbit of $\mathcal{B}\in\beth$ under the shift map $\sigma:\beth \rightarrow\beth$ has an accumulation point $\mathcal{B}^*\in \beth$. Then:
\begin{enumerate}
\item if the positive part of $\mathcal{B}^*$ is minimal, then the tail equivalence relation on the positive part of $\mathcal{B}$, and hence any Bratteli-Vershik transformation defined on it, admits a unique invariant probability measure. This gives a unique weight function $w^+$ so that the vertical flow on $S(\mathcal{B},w^\pm, \leq_{r,s})$ is uniquely ergodic \textbf{for any choice} of $w^-$ and $\leq_{r,s}$.
\item convergence along a subsequence is otherwise not sufficient for ergodicity: there is a Bratteli diagram $\mathcal{B}$ whose positive part is transitive with $\sigma^{n_k}(\mathcal{B})\rightarrow \mathcal{B}^*$ where $\mathcal{B}^*$ has a transitive positive part which admits a unique non-atomic probability invariant measure but $\mathcal{B}$ has a positive part which does not admit any non-atomic invariant probability measure. Moreover, the vertical flow on $S(\mathcal{B}^*,w^\pm, \leq_{r,s})$ is uniquely ergodic \textbf{for any choice} of $w^-$ and $\leq_{r,s}$.
\end{enumerate}
\end{theorem}
In other words, if the shift orbit of $\mathcal{B}$ has a ``good'' accumulation point then \textbf{any} flat surface constructed from it will have a uniquely ergodic vertical flow.
\begin{remark}
If the shift orbit does not have an accumulation point in $\beth$ then it may still be the case that the associated systems are uniquely ergodic. See Theorems 8.11 and 8.12 in \cite{LT}, for criteria of these type. These are systems given by Bratteli diagrams whose shift orbit ``diverges'', i.e., leaves every compact set of $\beth$, but does so ``slowly''.
\end{remark}
As pointed out in \cite{LT}, there already exist criteria for unique ergodicity in the literature (e.g. \cite{fisher,FFT,BKM09,BKMS2, FrankSadun:fusion}) for Bratteli-Vershik transformations defined through Bratteli diagrams. All of these criteria are from the Perron-Frobenius point of view. Theorem \ref{thm:main} adds to this list, but does so through the connection with flat surfaces. It would be interesting to see whether Theorem 1 could be proved using the Perron-Frobenius approach. The renormalization dynamics behind Theorem \ref{thm:main} are a generalization of Rauzy-Veech induction \cite{rauzy:CFE, rauzy:IET, veech:gauss}. Indeed, for a Bratteli diagram giving a translation surface of finite type, the shift dynamics $\sigma:\beth\rightarrow\beth$ coincide with Rauzy-Veech induction. This is the spirit of the renormalization scheme in \cite{bufetov:limitVershik, bufetov:limit}.

There are other parallels between properties of $\beth$ and those of moduli spaces of compact surfaces: the set of diagrams $\mathcal{B}$ corresponding to periodic surfaces is dense in $\beth$ (Proposition \ref{prop:density}) and so is the set of diagrams which give surfaces admitting pseudo-Anosov diffeomorphisms (Proposition \ref{prop:density2}). Moreover, the generic diagram gives a minimal and uniquely ergodic surface, i.e., the set of diagrams giving surfaces with those properties is residual (Proposition \ref{prop:residual}). 

There is a lot more to be done: the moduli space $\beth$ considered here is the coarsest imaginable and, besides a topology, it does not have much structure. For example, it is not clear whether $\beth$ is stratified in a way where different strata contain diagrams corresponding to surfaces with different topology (besides the strata coming from moduli spaces of compact surfaces), how the choices of orders $\leq_{r,s}$ can be incorporated into the moduli space, whether there is a ``natural'' shift-invariant probability measure on $\beth$ analogous to the Masur-Veech measures defined on moduli spaces of compact flat surfaces (or any other interesting shift-invariant probability measures), and what is the measure of diagrams with minimal positive parts. There are also some flat surfaces of finite area of interest, such that that coming from the infinite step billiard \cite{infinite-step}, which could only be constructed from Bratteli diagrams with countably infinite vertices at some level, which is outside the realm of the types of diagrams considered here. I hope to explore these questions in the future.

\begin{ack}
John Smillie was the first person to suggest to me that one should study the space of all Bratteli diagrams and the shift dynamics on it. The idea for this paper began with conversations with him, for which I am grateful. I also had many useful conversations with several people about the topics addressed in this paper, and I am grateful to all of them: Kathryn Lindsey, Barak Weiss, Pat Hooper, Boris Solomyak, Jon Aaronson, and Giovanni Forni. This work was partially supported by the NSF under Award No. DMS-1204008. Most of this paper was written while participating at ICERM's ``Dimension and Dynamics'' thematic semester. I am grateful for their support. 
\end{ack}
\section{Bratteli diagrams}
\label{sec:bratteli}

\begin{definition}
A \textbf{Bratteli diagram} is an infinite directed graph $B = (V,E)$ with
$$V = \bigsqcup_{k\in \mathbb{N}\cup \{0\}}V_k\hspace{.5in}  \mbox{ and } \hspace{.5in} E = \bigsqcup_{k\in \mathbb{N}}E_k$$
and maps $r,s:E\rightarrow V$ satisfying $r(E_k) = V_{k}$ and $s(E_k) = V_{k-1}$ for all $k\in\mathbb{N}$.
\end{definition}
We always assume that $|V_k|$ and $|E_k|$ are finite for all $k$. We may encode a Bratteli diagram through an infinite sequence of matrices $\{F_i\}_{i\in\mathbb{N}}$. The matrix $F_i = f^k_{v,w}$ is $|V_k|\times |V_{k-1}|$ and is given by $f^k_{v,w} = |\{e\in E_k: r(e) = v\mbox{ and } s(e) = w\}|$.

We denote by $X_B$ the set of all infinite paths starting at $V_0$:
$$X_B := \{ (e_1,e_2,\dots): e_i\in E_i,\,\, r(e_i) = s(e_{i+1})\mbox{ for all }i  \}.$$
The space is given the topology generated by clopen cylinder sets: given a finite collection $\bar{e} = (e_1,e_2,\dots, e_n)$, the cylinder set $C(\bar{e})$ is defined by
$$C(\bar{e}) = \{ (f_1,f_2,\dots ) \in X_B:  f_i = e_i \mbox{ for all }i\in\{1,\dots , n\}  \}.$$

Let $\leq_r$ be a partial order on each set $r^{-1}(v)$ for any $v\in V$. We call the pair $(B,\leq_r)$ an \textbf{ordered Bratteli diagramn}. For any $v\in V_k$ we denote by $S(v)$ to be the finite set
$$S(v) = \{(e_i,e_2,\dots,e_k ) :e_i\in E_i,\,\, r(e_i)=s(e_{i+1})\mbox{ for all $i\in\{1,\dots k\}$ and }  r(e_k) = v\}.$$
Note that $\leq_r$ gives an order to every set $S(v)$. A path $(e_1,e_2,\dots)$ is a \textbf{maximal path} if every $e_i$ is the maximal in $r^{-1}( r(e_i))$. We denote the set of maximal paths in $X_B$ by $X_{B}^+$. A path $(e_1,e_2,\dots)$ is a \textbf{minimal path} if every $e_i$ is the minimal in $r^{-1}( r(e_i))$. We denote the set of minimal paths in $X_B$ by $X_B^-$. Note that this depends on $\leq_r$.

\begin{definition}[Tail equivalence relation]
Let $Y\subset X_B$. The \textbf{tail equivalence relation} in $Y$ is the relation $\bar{e} \sim \bar{f}$, for $e,f\in Y$, if and only if there is a $k\in \mathbb{N}$ such that $e_i = f_i$ for all $i>k$.
\end{definition}
Note that a tail equivalence class which is finite has a unique minimal element and a unique maximal element. Let $[\bar{e}_1]$ be such a class and let $P(\bar{e}_1) = \{\bar{e}_1,\dots, \bar{e}_N\}$ be the elements in this class. In this case we have that the closure $\overline{P(\bar{e}_1)}$ in $Y$ is $P(\bar{e}_1)$. The union of elements making up a finite tail equivalence class will be refered to as a \textbf{periodic component}.

Tail equivalence classes which are infinite correspond to minimal components: let $[\bar{e}_1]$ be an infinite tail equivalence class and denote by $M(\bar{e}_1) = \{\bar{e}_1,\bar{e}_2,\dots\}$ the elements of this class. The closure $\overline{M(\bar{e}_1)}$ of $M(\bar{e}_1)$ in $X_B$ is a \textbf{minimal component}. Note that the definitions for periodic and minimal components do not depend on a choice of order $\leq_r$, but only on the tail equivalence relation.
\begin{definition}
Let $B$ be a Bratteli diagram and suppose that there exists a path $\bar{e}\in X_B$ such that $\overline{[\bar{e}]} = X_B$. Then $B$ and its tail equivalence relation are called \textbf{transitive}. If it is transitive and there is no periodic component, then $B$ and its tail equivalence relation are called \textbf{minimal}.
\end{definition}
\begin{remark}
An example of a Bratteli diagram which is transitive and not minimal is the Chacon middle third transformation: it contains a periodic component consisting of a single infinite path $\bar{e} = [\bar{e}]$. However, for any other path $\bar{e}^* \in X_B$ we have that $\overline{[\bar{e}^*]} = X_B$. This example will come up in \S \ref{subsec:example} (see also \cite[\S 6.2]{LT}).
\end{remark}
Let $\bar{e}= (e_1,e_2,\dots)\in X_B$ and $\bar{e}\not\in X_B^+$. Denote by $\ell(\bar{e})$ the smallest index $k$ such that $(e_1,e_2,\dots e_k)$ is non-maximal in $S(v)$, where $v = r(e_k)$. The \textbf{successor} of $\bar{e}\in X_B$ is the path $\bar{f} = (f_1,f_2,\dots)$ satisfying $f_i = e_i$ for all $i>\ell(\bar{e})$ and $(f_1,f_2,\dots ,f_{\ell(\bar{e})})$ is the successor of $(e_1,e_2,\dots ,e_{\ell(\bar{e})})$ in $S(v_{\ell(\bar{e})})$. This depends on $\leq_r$. For $\bar{e}\in X_B$ but $\bar{e}\not\in X_B^-$, the \textbf{predecessor} of $\bar{e}$ is the element $\bar{f}\in X_B$ such that $\bar{e}$ is the successor of $\bar{f}$. This also depends on $\leq_r$. We denote by $\rho_r^+(\bar{e})$ and $\rho_r^-(\bar{e})$ the successor and predecessor, respectively, of $\bar{e}$ and, recursively, $\rho^{j\pm 1}_r(\bar{e})$ is the successor/predecessor of $\rho^{j}_r(\bar{e})$. Let
$$\mathcal{X}_{B,r}^+ = X_B \backslash \bigcup_{k\in\mathbb{N}\cup\{0\}} \rho^{-k}_r (X_B^+).$$
\begin{definition}
Let $(B,\leq_r)$ be an ordered Bratteli diagram. The successor map $\rho_r^+:\bar{e}\mapsto \rho_r^+(\bar{e})$ defines the \textbf{Bratteli-Vershik} or \textbf{adic} transformation $\rho_r^+:\mathcal{X}_{B,r}^+\rightarrow X_B$.
\end{definition}
The transformation can be extended to all elements of periodic components: we declare that the successor of a maximal element in a periodic component is the minimal element of the periodic component. As such, the transformation restricted to a periodic component is periodic, and when we refer to a Bratteli-Vershik transformation we will always assume that it is also defined on a periodic component of $X_B$, where the tranformation depends on $\leq_r$.

If $|X_B^+| = |X_B^-| <\infty$, then we can extend the adic transformation to all of $X_B$ by choosing a bijection between $X_B^-$ and $X_B^+$ (respecting periodic components) and declaring the successor of $\bar{e}\in X^+_B$ to be the element in $X^-_B$ defined by this bijection. 
\begin{remark}
Note that if $B$ is transitive, then any adic transformation defined on $X_B$ will be trannsitive: there will be many orbits which are dense in $X_B$. If the tail equivalence relation is minimal, then any Bratteli-Vershik transformation will be minimal.
\end{remark}
We now recall the definition of a weight function \cite{LT}.
\begin{definition}
A \textbf{weight function} for a Bratteli diagram $B = (V,E)$ is a map $w:V_0\cup E\rightarrow [0,\infty)$ such that
\begin{enumerate}
\item for any $v\in V$ and two finite paths $(e_1,\dots, e_n)$ and $(f_1,\dots, f_n)$ with $r(e_n) = r(f_n)$ we have that
  $$w(s(e_1))\cdot \prod_{i=1}^nw(e_i) = w(s(f_1))\cdot \prod_{i=1}^nw(f_i),$$
\item for any $v\in V$
  $$\sum_{e\in s^{-1}(v)} w(e) = 1,$$
\item for any $\bar{e} = (e_1,e_2,\dots)\in X_B$ which does not belong to a periodic tail-equivalence class, we have that
  $$ \lim_{n\rightarrow \infty}w(s(e_1)) \cdot \prod_{i=1}^n w(e_i) = 0.$$
\end{enumerate}
It is a \textbf{positive weight function} if $w$ only takes values in $(0,\infty)$. It is a \textbf{finite} weight function if $\sum_{v\in V_0} w(v) <\infty$. It is a \textbf{probability} weight function if $\sum_{v\in V_0} w(v) = 1$. 
\end{definition}
If $w$ is a weight function for $B = (V,E)$, we can extend it to any $v\in V_k$ as follows. For any $v\in V_k$, $k\in\mathbb{N}$, let $\bar{e}=(e_1,e_2,\dots)$ be a path such that $r(e_k) = v$. Define
$$w(v) := w(s(e_1))\cdot\prod _{i=1}^k w(e_i).$$
This is independent of the path $(e_1,\dots, e_k)$ taken with $r(e_k) = v$ by (i) in the definition of a weight function.
\begin{remark}
\label{rem:measures}
A probability weight function defines a probability measure on $X_B$ which is invariant under the tail-equivalence relation. From this it follows that such measure is invariant for any adic transformation defined on $X_B$. Conversely, every invariant probability measure for an adic transformation on $X_B$ defines a probability weight function on $B$. 
\end{remark}
A \textbf{weighted} Bratteli diagram is a Bratteli diagram with a weight function, which we denote $(B,w)$. The following is standard (see \cite[Proposition 4.14]{LT}).
\begin{lemma}
\label{lem:existence}
Every Bratteli diagram $B$ admits a probability weight function.
\end{lemma}
\begin{remark}
In \S \ref{subsec:example} we construct examples of transitive diagrams which admit no non-atomic invariant probability measure but admit only an atomic ones.
\end{remark}
\begin{definition}
  A \textbf{fully ordered} Bratteli diagram is an ordered Bratteli diagram $(B,\leq_r)$ along with a partial order $\leq_s$ on $E\cup V$ so that any two edges $e,e'$ are comparable under $\leq_s$ if and only if $s(e) = s(e')$, $\leq_s$ is a total order on $V_k$, for any $k\geq 0$, and edges are not comparable with vertices. We denote fully ordered Bratteli diagrams by $(B,\leq_{r,s})$.
\end{definition}

\subsection{Bi-infinite diagrams}
\label{subsec:bi-inf}
We now review bi-infinite Bratteli diagrams, which will be the main actors in this play.
\begin{definition}
A \textbf{bi-infinite} Bratteli diagram is an infinite graph $\mathcal{B} = (\mathcal{V},\mathcal{E})$ which are grouped into disjoint sets
$$\mathcal{V} = \bigsqcup_{k\in \mathbb{Z}} \mathcal{V}_k\hspace{1in}\mbox{ and } \hspace{1in} \mathcal{E} = \bigsqcup_{k\in \mathbb{Z}\backslash \{0\}} \mathcal{E}_k$$
with associated range and source maps $r,s:\mathcal{E}\rightarrow \mathcal{V}$ such that $s(\mathcal{E}_k) = \mathcal{V}_{k-1}$ and $r(\mathcal{E}_k) = \mathcal{V}_{k}$ for $k\in\mathbb{N}$ and $s(\mathcal{E}_k) = \mathcal{V}_{k}$ and $r(\mathcal{E}_k) = \mathcal{V}_{k+1}$ for $k<0$.
\end{definition}
Following the convention in \cite{LT}, we will use calligraphic letters $\mathcal{B},\mathcal{E},\mathcal{V}, \mathcal{F}$ to denote bi-infinite diagrams and keep the uppercase letters $B,E,V,F$ for infinite diagrams defined defined in the previous section.

For $k<l$ We denote by $\mathcal{E}_{k,l}$ the set of paths from $\mathcal{V}_k$ to $\mathcal{V}_l$, that is,
$$\mathcal{E}_{k,l} = \{(e_{k+1},\dots, e_l)\in \mathcal{E}_{k+1}\times\cdots\times \mathcal{E}_l: r(e_i) = s(e_{i+1})\mbox{ for all } i \in \{k+1,\dots,l-1\} \}.$$

\begin{wrapfigure}{r}{0.25\textwidth}
  \centering
  \includegraphics[width=0.25\textwidth]{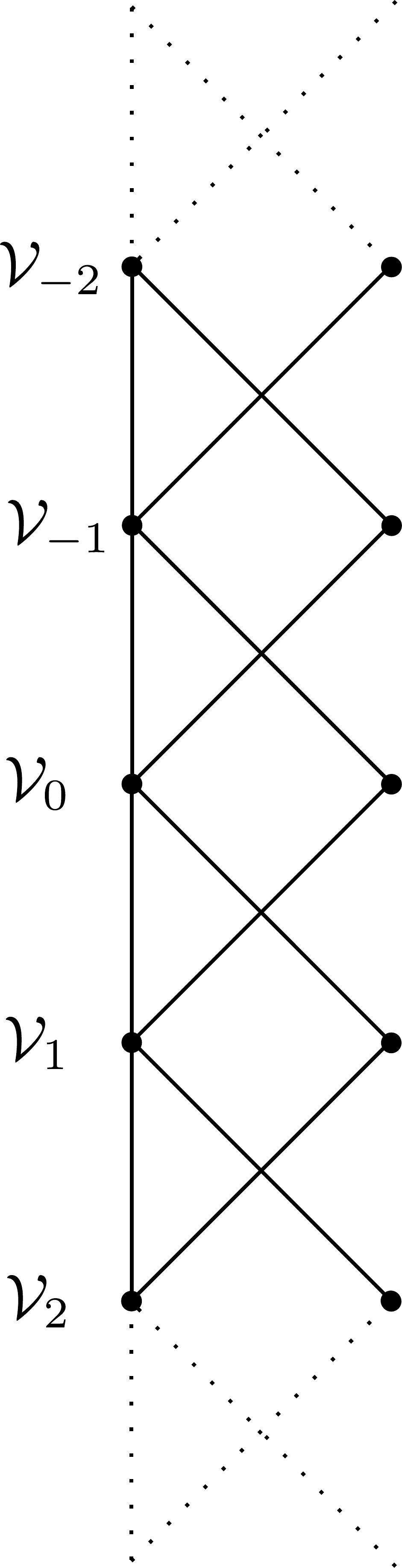}
\end{wrapfigure}

Given two Bratteli diagrams $B^- = (V^-,E^-)$ and $B^+ = (V^+,E^+)$ with $|V_0^-| = |V^+_0|$ we can create a bi-infinite Bratteli diagram $\mathcal{B}$ by \textbf{welding}: given a bijection $\beta:V_0^- \rightarrow V_0^+$, we identify $v\in \mathcal{V}_0$ with $v\in V_0^-$ and $\beta(v)\in V_0^+$, $\mathcal{E}_k^+$ and $\mathcal{V}_k^+$ with $E_k$ and $V_k$, respectively, for $k\in\mathbb{N}$ as well as $\mathcal{V}_{-k}$ and $\mathcal{E}_{-k}$ with $V_k^-$ and $E_k^-$, respectively. Welding is a way of combining two Bratteli diagrams $B^\pm$ to give a bi-infinite Bratteli diagram $\mathcal{B}$. Equivalently, the transition matrices $\{\mathcal{F}_k\}_{k\in\mathbb{Z}\backslash\{0\}}$ are given by $\mathcal{F}_{k} = F_k^+$ and $\mathcal{F}_{-k} = (F_k^-)^T$ for $k\in\mathbb{N}$.

The \textbf{positive part} of a bi-infinite Bratteli diagram $\mathcal{B} = (\mathcal{V},\mathcal{E})$ is the Bratteli diagram $B^+ = (V^+,E^+)$ where $E^+_k$ and $V_k^+$ are identified, respectively, with $\mathcal{V}_k$ and $\mathcal{E}_k$ for non-negative indices $k$. The \textbf{negative part} $B^-$ of $\mathcal{B}$ is similarly defined.

A \textbf{fully ordered} bi-infinite Bratteli diagram $(\mathcal{B},\leq_{r,s})$ is a bi-infinite Bratteli diagram with partial orders $\leq_r$ and $\leq_s$ defined as in the previous section. Note that when welding fully ordered Bratteli diagrams $(B^\pm, \leq_{r,s}^\pm)$ to produce a fully ordered bi-infinite Bratteli diagram $(\mathcal{B},\leq_{r,s})$ the $\leq_r^-$ orders on $B^-$ become the $\leq_s$ orders for the negative indices of $\mathcal{B}$. Likewise, $\leq_s^-$ becomes $\leq_r$ for negative indices after welding.

The space of infinite paths on $\mathcal{B}$ will be denoted by $X_\mathcal{B}$. It can be thought of as a subset of the product space of the space of paths of the positive and negative parts of $\mathcal{B}$:
$$X_\mathcal{B}  = \left\{ (\bar{e}^-,\bar{e}^+)\in X_{B^-}\times X_{B^+} : s(e_1^-) = s(e_1^+) \right\},$$
where $B^-$ and $B^+$ represent the negative and positive parts of $\mathcal{B}$, respectively.

A weighted bi-infinite Bratteli diagram $(\mathcal{B},w^\pm)$ is a bi-infinite Bratteli diagram along with two weight functions $w^+$ and $w^-$ for the negative and positive parts of $\mathcal{B}$, respectively. In these cases, we will always assume that $w^+$ is a probability weight function and that $w^-$ is finite and normalized such that
\begin{equation}
\label{eqn:probWeight}
\sum_{v\in\mathcal{V}_0} w^-(v)\cdot w^+(v) = 1.
\end{equation}
A \textbf{weighted, fully ordered, bi-infinite} Bratteli diagram is denoted by $(\mathcal{B},\leq_{r,s},w^\pm)$.
\subsection{The space of Bratteli diagrams}
\label{subsec:space}
Let $\mathfrak{M}_{m,n}$ be the set of all $m\times n$ matrices whose entries are in $\mathbb{N}\cup \{0\}$, where $m,n \in \mathbb{N}$. Let
$$\mathfrak{M} = \bigcup_{m,n} \mathfrak{M}_{m,n}$$
be the set of all non-negative integer valued matrices. Let 
$$\beth^+ \subset \mathfrak{M}^\mathbb{N}$$
be the set of all Bratteli diagrams $B$ identified as infinite paths in $\mathfrak{M}^\mathbb{N}$ with the restriction that if $B=(F_1,F_2,\dots) \in \beth^+$ with $F_k \in \mathfrak{M}_{m_k,n_k}$, then $m_i = n_{i+1}$ for all $i>0$. The set $\beth^+$ is given a topology by giving a basis of the topology of open cylinder sets
$$\mathfrak{U}(F_i,\dots, F_j) = \{B = (M_1,M_2,\dots) \in \beth^+ : M_k = F_k \mbox{ for all }i\leq k \leq j\}.$$
Since the main players are bi-infinite Bratteli diagrams $\mathcal{B}$ we now define an approriate space for them. Let 
\begin{equation}
\label{eqn:space}
\beth = \{(B^+,B^-) \in \beth^+ \times \beth^+ : \mbox{ if }F_{0}^\pm \in \mathfrak{M}_{m^\pm,n^\pm} \mbox{ then } n^+ = n^-\},
\end{equation}
which is topologized by the natural product topology. The set $\beth$ is the set of all bi-infinite Bratteli diagrams, where we think of it as subset of a product space, taking the positive and negative parts of each diagram $\mathcal{B}$ separately.
\begin{proposition}
\label{prop:residual}
The set of Bratteli diagrams with positive part which is minimal and uniquely ergodic is residual in $\beth$.
\end{proposition}
\begin{proof}
Let $\mathcal{\bar{\mathcal{F}}}$ be the $1\times 1$ matrix $2$. For $n\in\mathbb{N}$, define
$$\mathfrak{X}_n = \{\mathcal{B}\in\beth: \mathcal{F}_k \neq \bar{\mathcal{F}} \mbox{ for all }k\geq n  \}.$$
The set $\mathfrak{X}_n$ is nowhere dense in $\beth$. Indeed, suppose it is dense in some set $\mathfrak{U}\subset \beth$. By the topology of clopen sets of $\beth$, the open set $\mathfrak{U}$ is specified by cylinder sets: there is a countable collection of basic cylinder sets $\{\mathfrak{C}_i\}_i$ such that $\mathfrak{U} = \bigcup_i \mathfrak{C}_i$ and the basic cylinder sets are specified by finitely many positions.

If $\bar{\mathcal{B}} = (\dots,\bar{\mathcal{F}}_{-2}, \bar{\mathcal{F}}_{-1}, \bar{\mathcal{F}}_1,\bar{\mathcal{F}}_2,\dots)\in \mathfrak{U}$, then finitely many of its entries are specified. Let $\bar{\mathcal{F}}_m$ be the one specified with the greatest index in absolute value (without loss of generality, we are assuming $\mathfrak{U}$ specifies matrices with positive indices). Then any diagram $\bar{\mathcal{B}}^\heartsuit = (\dots, \bar{\mathcal{F}}_{-2}^\heartsuit, \bar{\mathcal{F}}_{-1}^\heartsuit,\bar{\mathcal{F}}_1^\heartsuit, \bar{\mathcal{F}}_2^\heartsuit,\dots)$ with by $\bar{\mathcal{F}}_i^\heartsuit = \bar{\mathcal{F}}_i$ for $0\leq |i| \leq m$ and $\bar{\mathcal{F}}$ for $i\geq 2m$ is in $\mathfrak{U}$. However, there is no sequence $\mathcal{B}_i\rightarrow \bar{\mathcal{B}}^\heartsuit$ with all $\mathcal{B}_i\in\mathfrak{X}_n$ since this would imply that for large enough $i$, the Bratteli diagram $\mathcal{B}_i$ would have the matrix $\bar{\mathcal{F}}$ as a transition matrix in an arbitrarily high place, which is incompatible with being in $\mathfrak{X}_n$. So $\mathfrak{X}_n$ is nowhere dense in $\beth$. The set
$$\mathfrak{S} = \bigcap_{n\in\mathbb{N}} \mathfrak{X}_n^c$$
is a countable intersection of dense open sets. Note that for any Bratteli diagram $\mathcal{B} = (\dots, \mathcal{F}_{-2},\mathcal{F}_{-1}, \mathcal{F}_1,\mathcal{F}_2,\dots)\in\mathfrak{S}$ we have that $\mathcal{F}_k = \bar{\mathcal{F}}$ for infinitely many $k>0$. In particular, for infinitely many $k>0$, $|\mathcal{V}_k|=1$, which implies both minimality and unique ergodicity (see Appendix \ref{sec:strict}). Thus, the set of Bratteli diagrams with minimal and ergodic positive parts is residual in $\beth$.
\end{proof}

\section{Flat surfaces and translation flows}
\label{sec:flat}
Let $S$ be a Riemann surface and $\alpha$ a holomorphic 1-form on $S$, and let $\bar{\Sigma}$ denote the zero set of $\alpha$. The pair $(S,\alpha)$ is a \textbf{flat surface}: integrating $\alpha$ locally away from $\bar{\Sigma}$ gives coordinate charts to $S$ which are flat. That is, it gives $S$ charts $\{(\phi_i,U_i)\}$ of the form $\phi_{i}:U_i\rightarrow \mathbb{R}^2$ with transition functions of the form $\phi_i\circ \phi_j^{-1} : z\mapsto z+c_{i,j}$ for some $c_{i,j}\in\mathbb{R}^2$. The area of the flat surface is $\frac{2}{i}\int_S \alpha\wedge\bar\alpha$. We will be interested only in surfaces of finite area, so we will always in fact assume that $\frac{2}{i}\int_S \alpha\wedge\bar\alpha = 1$.

We do note assume that $S$ is compact. Let $\bar{S}$ denote the metric completion of $S$ with respect to the flat metric and let $\Sigma = \bar{\Sigma} \cup (\bar{S}\backslash S)$. We will refer to $\Sigma$ as the \textbf{singularity set} of $(S,\alpha)$.

A flat surface comes with two distiguished directions, the \textbf{vertical and horizontal} directions. They are given by locally (away from $\Sigma$) by the distributions $\ker \Re(\alpha)$ and $\ker \Im(\alpha)$, respectively. These distributions are integrable and they define the vertical and horizontal foliations, $\mathcal{F}^v_\alpha$ and $\mathcal{F}^h_\alpha$. The unit speed parametrizations of these foliations are called the \textbf{vertical and horizontal flows} and they are denoted by $\varphi_t^Y$ and $\varphi_t^X$, respectively. These flows are not defined for all time for all points: a point on a singular leaf (a leaf which includes a point in $\Sigma$ as a limit point) of the foliation has an orbit which is defined only for finite time and it ends when the orbit goes to $\Sigma$. Both of these flows preserve the Lebesgue measure $\omega_\alpha = \Re(\alpha)\wedge \Im(\alpha)$.

Let $(S,\alpha)$ be a flat surface. The \textbf{Teichm\"{u}ller deformation} of $(S,\alpha)$ is the one parameter family of flat surfaces $g_t(S,\alpha):=(S,\alpha_t)$, where $\alpha_t  = e^t\Re(\alpha) + ie^{-t}\Im(\alpha)$. The effect this deformation has on the original surface is that it contracts the vertical foliation by $e^{-t}$ while it expands the horizontal foliation by $e^t$. The area of the surface remains unchanged.
\subsection{Surfaces from diagrams}
\label{subsec:SirfDiag}
We recall the construction from \cite[\S 6]{LT} where, given a weighted, fully ordered bi-infinite Bratteli diagram $(\mathcal{B},w^\pm,\leq_{r,s})$ a flat surface $S(\mathcal{B},w^\pm,\leq_{r,s})$ of area 1 can be constructed. Appendix \ref{sec:fib} contains an explicit example of the construction from \cite[\S 6]{LT}. The reader should consult \cite[\S 6]{LT} for the full details of the construction as well as more examples. 

Let $(B^+,w^+,\leq_{r,s})$ and $(B^-,w^-,\leq_{r,s})$ be, respectively, the weighted, fully ordered Bratteli diagrans obtained from the positive and negative parts of the weighted, fully ordered bi-infinite Bratteli diagram $(\mathcal{B},w^\pm, \leq_{r,s})$ as described in \S \ref{subsec:bi-inf}. As detailed in \cite[\S 6.1]{LT}, there exist maps $T^+:[0,1]\rightarrow[0,1]$ and $T^-:[0,b^-]\rightarrow [0,b^-]$, where $b^-$ will be defined by a condition of the type (\ref{eqn:probWeight}), which are piecewise, orientation-preserving isometries, that is, interval exhange transformations. These interval exchange transformations are measurably conjugate to the Bratteli-Vershik transformations defined by the positive/negative parts of $\mathcal{B}$, with the weight functions $w^\pm$ giving the invariant measures (see Remark \ref{rem:measures}). These transformations are constructed using cutting-and-stacking using the information given by the Bratteli diagrams $(B^\pm,w^\pm,\leq_{r,s})$.

Define the $|\mathcal{V}_0|$ sets $R_1,\dots, R_{|\mathcal{V}_0|}$ by
\begin{equation}
\label{eqn:rectangles}
R_i = \left[\sum_{i<j}w^+(v_j),\sum_{i\leq j}w^+(v_j)\right]\times \left[\sum_{i<j}w^-(v_j),\sum_{i\leq j}w^-(v_j)\right].
\end{equation}
We define $\Sigma\subset \bigcup_i R_i$ to be the points in $\bigcup_i\partial R_i$ which correspond to discontinuities of $T^+$ or $T^-$. The flat surface $S(\mathcal{B},w^\pm,\leq_{r,s})$ is constructed by giving identifications to the collection of edges of the $R_i$:
\begin{equation}
\label{eqn:RecIds}
 S(\mathcal{B},w^\pm,\leq_{r,s}) = \left(\bigcup_{i=1}^{|\mathcal{V}_0|} R_i\backslash \Sigma   \right) / \sim ,
\end{equation}
where $ \sim$ is the relation given by the following identifications on the edges on $\left(\bigcup_i \partial R_i  \right)\backslash \Sigma$. For a point $(x,y)$ on a top edge of $\bigcup_i R_i\backslash \Sigma$, then $(x,y)\sim (T^+(x), y_{T^+(x)})$, where $y_{T^+(x)}$ is the $y$-coordinate of the bottom edge of the rectangle where $T^+(x)$ lies. Likewise,for a point $(x,y)$ on a right edge of $\bigcup_i R_i\backslash \Sigma$, then $(x,y)\sim (x_{T^-(y)}, T^-(y))$, where $x_{T^-(y)}$ is the $x$-coordinate of the left edge of the rectangle where $T^-(y)$ lies.

In the definition above of the surface, we did not give an explicit holomorphic 1-form $\alpha$ to give this surface. However, this 1-form is implicitly defined since the surface does has change of charts given by translations. For the translation surface $S(\mathcal{B},w^\pm,\leq_{r,s})$, the vertical flow on $S(\mathcal{B},w^\pm,\leq_{r,s})$ is a special flow over the interval exchange transformation $T^+$ with constant roof functions given by $w^-(v)$, $v\in\mathcal{V}_0$. Likewise, the horizontal flow corresponds to a special flow over the exchange $T^-$ with roof functions given by $w^+(v)$, $v\in\mathcal{V}_0$. The weight functions $w^\pm$ correspond to the transverse measures to the vertical and horizontal flows which are invariant by the flows: they are given by $\Re(\alpha)$ for $w^+$ and $\Im(\alpha)$ for $w^-$. The Lebesgue measure on $S(\mathcal{B},w^\pm,\leq_{r,s})$ is the product measure of these measures ($\omega_\alpha = \Re(\alpha)\wedge\Im(\alpha)$), and we denote it by $\omega_{(\mathcal{B},w^\pm, \leq_{r,s})}$. As constructed, if the weight functions $w^\pm$ satisfy (\ref{eqn:probWeight}), then $S(\mathcal{B},w^\pm,\leq_{r,s})$ has area 1.

The construction of the surface $S(\mathcal{B},w^\pm, \leq_{r,s})$ from the diagram $(\mathcal{B},w^\pm, \leq_{r,s})$ was done through a series of cutting and stacking operations described in \cite[\S 6]{LT}. Another way of thinking about this is to start the $|V_0|$ tori corresponding to the rectangles in (\ref{eqn:rectangles}) each with its top/bottom and left/right sides identifies. We can think of the surface $S(\mathcal{B},w^\pm, \leq_{r,s})$ as the limiting object we get by making, for every level of the positive (negative) part of the Bratteli diagram $\mathcal{B}$, a series of incisions on the tori somewhere on the collection top (left) edges and identifying them with incisions made somewhere on the collection of bottom (right) edges of the rectangles. The size of the incisions are dictated by $w^+$ ($w^-$) while the identifications are defined by $\leq_{r,s}$ via the cutting and stacking operations. Therefore, for every finite $k$, considering the the surface obtained by starting with $|V_0|$ tori and making such incisions up to level $k$, one obtains a sequence $S^k(\mathcal{B},w^\pm, \leq_{r,s})$ of surfaces of finite genus. The vertical flow on any one of these surfaces is completely periodic: almost every point flows for a finite time before coming back to itself and the surface decomposes into a disjoint union of cylinders corresponding to families of periodic orbits. This idea leads to the following.
\begin{proposition}
\label{prop:density}
The set $\mathfrak{P}$ of Bratteli diagrams $\mathcal{B}$ for which any surface $S(\mathcal{B},w^\pm,\leq_{r,s})$ has a completely periodic vertical and horizontal foliation is dense in $\beth$.
\end{proposition}
\begin{proof}
Let $\mathcal{B}\in\beth$ and let $S(\mathcal{B},w^\pm, \leq_{r,s})$ be  the surface from the diagram $(\mathcal{B},w^\pm, \leq_{r,s})$ for some choice of $w^\pm$ and $\leq_{r,s}$. Define the sequence of Bratteli diagrams $(\mathcal{B}_i,w^\pm_i,\leq_{r,s}^i)$ as follows. Let $\mathcal{F}_k$ be the $k^{th}$ matrix corresponding to $\mathcal{B}$. The $k^{th}$ transition matrix of $\mathcal{B}_i$, $\mathcal{F}_k^i$, is defined as
$$ \mathcal{F}_k^i = \left \{ \begin{array}{ll}  \mathcal{F}_k &\mbox{ if } |k| \leq i, \\
\mathrm{Id}_{|\mathcal{V}_{i}|}&\mbox{ if  } k>i \\
\mathrm{Id}_{|\mathcal{V}_{-i}|}&\mbox{ if  } k<-i
\end{array}\right. ,$$
where $\mathrm{Id}_n$ is the $n\times n$ identity matrix. Every such Bratteli diagram $\mathcal{B}_i$ decomposes into finitely many is periodic components for the tail equivalence relation for both its positive and negative parts. We can give weights and ordering to each $\mathcal{B}_i$ as follows. For $|k|\leq i$, the weigth for an edge $e \in \mathcal{E}_k^i$ is $w_i^\pm(e) = w^\pm(e)$. Since $\mathcal{B}_i$ is periodic, this determines all the weights in $\mathcal{B}_i$. The orderings for $\mathcal{B}_i$ are obtained in a similar way: for $|k|\leq i$, we just copy the orderings at every vertex $v\in \mathcal{V}_k$ to its copy $\mathcal{V}_k^i$. As such, we have the convergence $\mathcal{B}_i\rightarrow \mathcal{B}$ with any surface $S(\mathcal{B}_i,w_i^\pm, \leq^i_{r,s})$ having completely periodic horizontal and vertical flows.
\end{proof}
As such, the weighted, ordered Bratteli diagram $(\mathcal{B},w^\pm, \leq_{r,s})$ can be seen as the limit of \textbf{periodic} Bratteli diagrams $(\mathcal{B}_i,w^\pm_i, \leq_{r,s}^i)$, and therefore $S(\mathcal{B},w^\pm, \leq_{r,s})$ can be also considered as a limit of $S(\mathcal{B}_i,w_i^\pm, \leq_{r,s}^i)$. The dynamical system defined by the vertical flow on $S(\mathcal{B},w^\pm, \leq_{r,s})$ can be seen as a limit of periodic systems.
 When $S(\mathcal{B},w^\pm, \leq_{r,s})$ is a flat surface of finite genus, this fact is already known: $S(\mathcal{B},w^\pm, \leq_{r,s})$ belongs to a stratum $\mathcal{H}(\kappa)$ in the moduli space of flat surfaces and it is well known that periodic surfaces are dense in the moduli space. Therefore any sequence in $\mathcal{H}(\kappa)$ consisting of completely periodic surfaces which are limiting to $S(\mathcal{B},w^\pm, \leq_{r,s})$ is a periodic approximation to the vertical flow of $S(\mathcal{B},w^\pm, \leq_{r,s})$. Therefore, what is surprising, by the above remarks, is that even if $S(\mathcal{B},w^\pm, \leq_{r,s})$ is of infinite genus, there are sequences $S^k$ of completely periodic surfaces of finite genus limiting to $S(\mathcal{B},w^\pm, \leq_{r,s})$. As such, we could imagine that periodic surfaces are also dense in the bigger ``moduli space'' $\beth$ which include both finite genus and infinite genus surfaces.
\subsection{Renormalization}
\label{subsec:renorm}
We recall the renormalization procedure obtained by shifting introduced in \cite{LT}. Let $(\mathcal{B},w^\pm, \leq_{r,s})$ be a weighted, ordered Bratteli diagram. In this section we define a sequence of maps 
$$\mathcal{R}_k: S(\mathcal{B},w^\pm,\leq_{r,s})\longrightarrow S(\mathcal{B}'_k, w^\pm_k, \leq^k_{r,s}) $$
which take flat surfaces constructed as in the previous section to other surfaces of the same type.

For a weighted, fully ordered Bratteli diagram $(\mathcal{B},w^\pm),\leq_{r,s}$, let $h^0=\{h_1^0,\dots, h_{|\mathcal{V}_0|}^0\}\in\mathbb{R}^{|\mathcal{V}_0|}$ be the vector where every entry is defined by $h_v^0 = w^-(v)$ (recall that the vertices in $\mathcal{V}_0$ are ordered, so we identify each vertex $v$ with its corresponding index in $\{1,\dots, |\mathcal{V}_0|\}$). Define recursively the height vectors $h^k = \mathcal{F}_k h^0$ for $k>1$. Define $\ell_v^k = w^+(v)$ for a specific $v\in \mathcal{V}_k$, $k\in\mathbb{N}$.

The data defining the maps $\mathcal{R}_k$ will be related as follows: for $\mathcal{B} = (\mathcal{V},\mathcal{E})$, $\mathcal{B}'_k = (\mathcal{V}',\mathcal{E}')$ is obtained by shifting $\mathcal{B}$:  $\mathcal{V}'_i = \mathcal{V}_{i+k}$ and $\mathcal{E}'_i = \mathcal{E}_{i+k}$ along with their orders $\leq_{r,s}$ and $w^\pm_k = e^{\pm t_k}w^\pm$, where the $t_k$ belong to the sequence of renormalization times
\begin{equation}
\label{eqn:rTimes}
t_k \equiv -\log\left( \sum_{v\in \mathcal{V}_k}\ell^k_v \right) = -\log\left( \sum_{v\in \mathcal{V}_k}w^+(v) \right)
\end{equation}
for $k>0$. Note that if the positive part of $\mathcal{B}$ has a periodic component under the tail equivalence relation to which $w^+$ assigns a positive weight, then $\sup_k t_k <\infty$. Otherwise, if all periodic components are given zero weight, or if there are no periodic components, we have that $t_k\rightarrow\infty$.

The renormalized heights and widths are obtained from (\ref{eqn:rTimes}) by
\begin{equation}
\label{eqn:rHV}
\bar{h}^k_v = e^{-t_k}h^k_v\hspace{.2in}\mbox{ and } \hspace{.2in}\bar{\ell}^k_v = e^{t_k} \ell^k_v
\end{equation}
for any $v\in \mathcal{V}_k$ and $k\in\mathbb{N}$. Until further notice, we shall assume that the positive part of $\mathcal{B}$ is aperiodic so that $t_k\rightarrow \infty$.

Let $(\mathcal{B},w^\pm, \leq_{r,s})$ be a weighted, ordered Bratteli diagram, $S(\mathcal{B},w^\pm, \leq_{r,s})$ be the flat surface constructed obtained from it through the construction in \S \ref{subsec:SirfDiag}. Let $S_t(\mathcal{B},w^\pm,\leq_{r,s}) = g_tS(\mathcal{B},w^\pm,\leq_{r,s})$ be the surface obtained by deforming $S(\mathcal{B},w^\pm,\leq_{r,s})$ using the Teichm\"{u}ller deformation. Consider the surface $S_{t_1}(\mathcal{B},w^\pm, \leq_{r,s})$, for $t_1$ defined in (\ref{eqn:rTimes}).

Choose some vertex $v_i \in \mathcal{V}_0$. By our deformation of the surface, the interior of every deformed rectangle $g_{t_1}R_i$ in (\ref{eqn:rectangles}) corresponding to the vertex $v_i$ in $\mathcal{V}_0$ is isometric to $(0,\bar{\ell}_i^0)\times (0,\bar{h}^0_i)$. We cut the rectangle $g_{t_1}R_i$ (associated to the tower over $v_i$) into sub-rectangles of width $\bar{\ell}^0_{v_i}w(e)$ and height $\bar{h}^0_{v_i}$ using the placement choice $\leq_s$ on $v_i$ for every $e$ is an edge with $s(e)=v_i$. Doing this for every vertex $v_i\in \mathcal{V}_0$ we have $|E_1|$ subrectangles corresponding to edges in $E_1$ which were obtained as subrectangles of the $R_j$.

Now the sub-rectangles are stacked into $|\mathcal{V}_1|$ new towers using the orders given by the order in $r^{-1}(v)$. For some $v\in \mathcal{V}_1$, let $(e_1,\dots,e_n)$ be the ordered set of edges in $r^{-1}(v)$. For all $i\in \{1,\dots, |r^{-1}(v)|-1\}$, we identify the interior of the top of the sub-rectangle corresponding to the edge $e_i$ to the interior of the bottom edge of the sub-rectangle corresponding to the edge $e_{i+1}$.  Denote the surface obtained by the process of deforming and cutting and stacking described above as $DS(\mathcal{B},w^\pm,\leq_{r,s})$ and define
\begin{equation}
\label{eqn:renMap}
\mathcal{R}:S(\mathcal{B},w^\pm,\leq_{r,s})\longrightarrow DS(\mathcal{B},w^\pm,\leq_{r,s})
\end{equation}
to be the map taking one surface to the other by this process. We point out that the map between $S_{t_1}(\mathcal{B},w^\pm,\leq_{r,s})$ and $S(\mathcal{B}', w^\pm_1, \leq^1_{r,s})$ is an isometry: the cutting and stacking does not change the flat metric in any way.

\begin{definition}[Shifting]
The diagram $(\mathcal{B}',w^\pm_1,\leq'_{r,s})$ is the \emph{shift} of $(\mathcal{B},w^\pm,\leq_{r,s})$ if it can be constructed as follows. $\mathcal{V}'_{i} = \mathcal{V}_{i+1}$ for all $i\in\mathbb{Z}$ and $\mathcal{E}'_i=\mathcal{E}_{i+1}$ for all $i\neq -1$. For $i=-1$, $\mathcal{E}'_{-1} = \mathcal{E}_1$. As such, there is a bijection $\sigma: \mathcal{B}'\rightarrow \mathcal{B}$ corresponding to this shift. 

Now we define $w^\pm_1:\mathcal{V}_0'\cup \mathcal{E}'\rightarrow (0,1)$ to be the weight function obtained from $(\mathcal{B},w^\pm)$ as follows: for $v\in \mathcal{V}_0'$, $w^\pm_1(v) = e^{\pm t_1}w^\pm(\sigma^{-1}(v))$, $w^+_1(e) = w^+(\sigma^{-1}(e))$ for any $e\in \sigma(\mathcal{E}^+)\backslash \mathcal{E}'_{-1}$, and $w^-_1(e) = w^-(\sigma^{-1}(e))$ for any $e\in \sigma(\mathcal{E}^-)$. Let $\leq_{r,s}'$ be defined on $\mathcal{B}'$ by $e\leq_{r,s}' f$ if and only if $\sigma(e) \leq_{r,s} \sigma(f)$ on $\mathcal{B}$.
\end{definition}
We will denote by $\sigma(\mathcal{B},w^\pm,\leq_{r,s})$ the shift of $(\mathcal{B},w^\pm,\leq_{r,s})$, and by $\sigma^k$ the process of shifting $k$ times. It is straight forward to check from the definition that if $(\mathcal{B},w^\pm,\leq_{r,s})$ is a weighted, ordered Bratteli diagram, then so is $\sigma(\mathcal{B},w^\pm,\leq_{r,s})$. The following proposition, which was proved in \cite{LT}, showcases a crucial property of the construction.
\begin{proposition}[Functoriality]
\label{prop:shift}
Let $(\mathcal{B},w^\pm, \leq_{r,s})$ be a weighted, ordered Bratteli diagram. Then $S(\sigma(\mathcal{B},w^\pm, \leq_{r,s})) $ is isometric to $g_{t_1}S(\mathcal{B},w^\pm, \leq_{r,s})$.
\end{proposition}

The shift $\sigma$ on a Bratteli diagram yields a sequence of surfaces 
$$S_k(\mathcal{B},w^\pm, \leq_{r,s}) \equiv S(\sigma^k(\mathcal{B},w^\pm, \leq_{r,s}))$$
 which are obtained as a shift on the starting Bratteli diagram $B$ and by rescaling the weights in the shifted diagram by the appropriate quantities.
 In a slight abuse of notation, we will denote by $\mathcal{R}_k$ the map satisfying
\begin{equation}
\label{eqn:Rmaps} 
\mathcal{R}_k(S(\mathcal{B},w^\pm, \leq_{r,s})) = S(\sigma^k(\mathcal{B},w^\pm, \leq_{r,s})),
\end{equation}
which is obtained through composition of maps of the type defined in (\ref{eqn:renMap}). By Proposition \ref{prop:shift}, for each $k$ the surface $S_k(\mathcal{B},w^\pm, \leq_{r,s})$ is obtained by deforming $S(\mathcal{B},w^\pm, \leq_{r,s})$ for time $t_k$ and then cutting and stacking. We will denote this sequence of surfaces by 
\begin{equation}
\label{eqn:surfSeq}
S(\mathcal{B}_k,w^\pm_k,\leq_{r,s}^k) := S(\sigma^k(\mathcal{B},w^\pm,\leq_{r,s})).
\end{equation}

\begin{definition}
\label{def:stationary}
A Bratteli diagram $\mathcal{B}$ is called \textbf{stationary} if there is a $j\in\mathbb{N}$ such that $\mathcal{F}_{k+nj} = \mathcal{F}_{k}$ for all $k,n\in\mathbb{Z}$. In other words, the transition matrices $\mathcal{F}_k$ of a stationary diagram repeat themselves every $j$ entries. A fully ordered Bratteli diagram $(\mathcal{B},w^\pm)$ is stationary if the orders $\leq_{r,s}$ are also repeat themselves after $j$ entries.
\end{definition}
For a flat surface $S(\mathcal{B},w^\pm,\leq_{r,s})$ constructed from a fully ordered, stationary Bratteli diagram, (\ref{eqn:Rmaps}) defines a \textbf{pseudo-Anosov} map $\mathcal{R}_j:S(\mathcal{B},w^\pm,\leq_{r,s})\rightarrow S(\mathcal{B},w^\pm,\leq_{r,s})$, that is, a homeomorphism of the surface which is uniformly expanding on the horizontal direction and uniformly contracting in the vertical one on the complement of $\Sigma$. The following can be proved by a density argument similar to that of the proof of Proposition \ref{prop:density}.
\begin{proposition}
\label{prop:density2}
The set $\mathfrak{A}$ of Bratteli diagrams $\mathcal{B}$ for which $S(\mathcal{B},w^\pm,\leq_{r,s})$ admits a pseudo-Anosov diffeomorphism is dense in $\beth$.
\end{proposition}
\section{Proof of Theorem \ref{thm:main}}
\label{sec:proof}
In this section we prove Theorem \ref{thm:main}. To do so, we will invoke the following theorem from \cite{rodrigo:ergodicity}.
\begin{theorem}
\label{thm:integrability}
Let $(S,\alpha)$ be a flat surface of finite area. Suppose that for any $\eta>0$ there exist a function $t\mapsto \varepsilon(t)>0$, a one-parameter family of subsets
$$S_{\varepsilon(t),t} = \bigsqcup_{i=1}^{C_t}S_t^i$$ 
of $S$ made up of $C_t < \infty$ path-connected components, each homeomorphic to a closed orientable surface with boundary, and functions $t\mapsto \mathcal{D}_t^i>0$, for $1\leq i \leq C_t$, such that for 
$$\Gamma_t^{i,j} = \{\mbox{paths connecting }\partial S_t^i \mbox{ to }\partial S_t^j\}$$
and
\begin{equation}
\label{eqn:systole}
\delta_t = \min_{i\neq j} \sup_{\gamma\in\Gamma_t^{i,j} }\mbox{dist}_t(\gamma,\Sigma)
\end{equation}
the following hold:
\begin{enumerate}
\item  $\omega_\alpha(S\backslash S_{\varepsilon(t), t}) < \eta$ for all $t>0$,
\item $\mbox{dist}_t(\partial S_{\varepsilon(t),t},\Sigma) > \varepsilon(t)$ for all $t>0$,
\item the diameter of each $S_t^i$, measured with respect to the flat metric on $(S,\alpha_t)$, is bounded by $\mathcal{D}_t^i$ and
\begin{equation}
\label{eqn:integrability2}
\int_0^\infty \left( \varepsilon(t)^{-2}\sum_{i=1}^{C_t}\mathcal{D}_t^i + \frac{C_t-1}{\delta_t}\right)^{-2}\, dt = +\infty.
\end{equation}
\end{enumerate}
Moreover, suppose the set of points whose translation trajectories leave every compact subset of $S$ has zero measure. Then the translation flow is ergodic.
\end{theorem}
\begin{definition}
Let $\mathcal{B} = (\mathcal{E},\mathcal{V})$ be a Bratteli diagram and let $k\in\mathbb{Z}$. Suppose that for any $v,w\in \mathcal{V}_k,$ there exists a $k_{v,w} = \Delta^k_\mathcal{B}(v,w)$ and paths $p, q \in \mathcal{E}_{k,k + k_{v,w}}$ with $s(p) = v, s(q) = w$ and $r(p) = r(q)$. For any $k\in\mathbb{Z}$, the \textbf{metamour function} is the function $\Delta_{\mathcal{B}}^k: \mathcal{V}_k\times \mathcal{V}_k\rightarrow \mathbb{N}\cup\{\infty\},$  $(v,w)  \mapsto k_{v,w} = \Delta^k_{\mathcal{B}}(v,w)\in  \mathbb{N}\cup\{\infty\}$ defined, for $k\in\mathbb{Z}$ and $v,w\in\mathcal{V}_k$, as the minimum $k_{v,w}$ for which there exists paths $p, q \in \mathcal{E}_{k,k + k_{v,w}}$ with $s(p) = v, s(q) = w$ and $r(p) = r(q)$. If no such paths exist,  $\Delta^k_\mathcal{B}(v,w) = \infty$. We also define
$$\Delta_\mathcal{B}^+(k) := \max_{v,w \in \mathcal{V}_k}\Delta_{\mathcal{B}}^k(v,w).$$
\end{definition}
The motivating idea for the above definition is the following. In the proof of the main theorem, given a level $\mathcal{V}_k$, $k\in\mathbb{N}\cup\{0\}$, we will need to find connecting paths between any two elements $v,w\in\mathcal{V}_k$. To do this, the metamour function gives the minumum number of levels down from level $k$ one must go to guarantee that we can find paths connecting vertices $v$ and $w$ going through some other elements in $\mathcal{V}_{k+\Delta^k_\mathcal{B}(v,w)}$. If $\Delta_\mathcal{B}^k(v,w) = \infty$, then this means that this is not possible, i.e., that there are vertices $v,w\in\mathcal{V}_K$ which are in (forwardly) disconnected components. This happens when there are several minimal components, or when there is a periodic component.
\begin{lemma}
\label{lem:practicalTunnel}
Let $\mathcal{B}$ be a bi-infinite Bratteli diagram. If the positive part of $\mathcal{B}$ is transitive then we have that $\Delta_\mathcal{B}^+(0) = \max_{v,w \in \mathcal{V}_0}\Delta_{\mathcal{B}}^0(v,w) < \infty$.
\end{lemma}
\begin{proof}
Let $B^+$ be the positive part of $\mathcal{B}$ and let $\bar{e}^*\in X_{B^+}$ be a representative of the dense tail equivalence relation, i.e., $\overline{[\bar{e}^*]}= X_{B^+}$.

Since the class of $\bar{e}^* = (e_1^*,e_2^*,\dots)$ is dense in $X_{B^+}$, for any $v\in V_0$ and $m\in \mathbb{N}$, there exists a $m^*_{v,m}\geq m$ and path $\bar{e}^{v,m} = (e_1^{v,m},e_2^{v,m},\dots)\in X_{B^+}$ with $s(\bar{e}^{v,m}) = v$ such that $e_k^{v,m} = e_k^*$ for all $k\geq m^*_{v,m}$. Define $\varpi = \max_{v\in V_0} m^*_{v,1}$ and consider the $|V_0|$ paths $\bar{e}^{1,1},\dots, \bar{e}^{|V_0|,1} \in X_{B^+}$, that is, paths for $m = 1$ and starting at all possible vertices in $V_0$. As such, we have that $e_k^{v,1} = e_k^*$ for all $k\geq \varpi$.

Now given any $v,w\in V_0$, we follow the path $\bar{e}^{v,1}$ down until $k = \varpi$, where we know that $r(e_\varpi^{v,1}) = r(e_\varpi^{w,1})$, and we follow $\bar{e}^{w,1}$ back up to $w\in V_0$. It follows that $\Delta_{\mathcal{B}}^0(v,w) \leq \varpi$. 
\end{proof}
\begin{lemma}
  \label{lem:minimality}
  Suppose that for $\mathcal{B}\in\beth$ there is a subsequence $k_i\rightarrow\infty$ such that $\sigma^{k_i}(\mathcal{B})\rightarrow \mathcal{B}^*$ for some $\mathcal{B}^*\in\beth$ whose positive part is minimal. Then the positive part of $\mathcal{B}$ is minimal.
\end{lemma}
\begin{proof}
  Denote $\Delta \equiv \Delta_{\mathcal{B}^*}^+(0)$ which is finite by Lemma \ref{lem:practicalTunnel}. Since the space $\beth$ is topologized by cylinder sets coming from the coding of $\mathfrak{M}$ into a countable alphabet, we can assume without loss of generality that $\mathcal{F}^{k_i}_l = \mathcal{F}^*_l$ for $-i \leq l \leq i$. Therefore, it follows that $\Delta^+_\mathcal{B}(k_i) = \Delta_{\mathcal{B}^*}^+(0)$ for all $i>\Delta$.

Let $\bar{x} = (x_1,x_2,\dots)\in X_{\mathcal{B}^+}$ and $\bar{x}' = (x_1',x_2',\dots)\in X_{\mathcal{B}^+}$ be two paths. To show minimality, we need to show that for any $K>0$ there is a path $\bar{y} = (y_1,y_2,\dots)\in X_{\mathcal{B}^+}$ and an integer $K'>0$ with the property that $y_j = x_j$ for all $j\leq K$ and $y_j = x_j'$ for all $j>K'$.

Let $K>0$ be given and let $i_K>\Delta$ be an integer such that $k_{i_K} > K$. Let $v_{\bar{x}} = r(x_{k_{i_K}})\in \mathcal{V}_{k_{i_K}}$ and $v_{\bar{x}'} = r(x_{k_{i_K}}')\in \mathcal{V}_{k_{i_K}}$. By the convergence $\sigma^{k_i}(\mathcal{B})\rightarrow \mathcal{B}^*$, that $i_K>\Delta$, and by Lemma \ref{lem:practicalTunnel}, $\Delta_{\mathcal{B}}^{k_{i_K}}(v_{\bar{x}}, v_{\bar{x}'})$ is finite. Therefore, there are paths $p_{\bar{x}},q_{\bar{x}'} \in\mathcal{E}_{k_{i_K}, k_{i_K}+\Delta_{\mathcal{B}}^{k_{i_K}}(v_{\bar{x}}, v_{\bar{x}'})}$ with $s(p_{\bar{x}}) = v_{\bar{x}}$, $s(q_{\bar{x}'}) = v_{\bar{x}'}$ and $r(p_{\bar{x}}) = r(q_{\bar{x}'})$.

So there is a path $\bar{y} = (y_1,y_2,\dots)\in X_{\mathcal{B}^+}$ with the following properties: $y_j = x_j$ for all $j\in \{1,\dots, k_{i_K} \}$, $\bar{y}$ coincides with $q_{\bar{x}'}$ for $j\in\{k_{i_K}+1,k_{i_K}+\Delta_{\mathcal{B}}^{k_{i_K}}(v_{\bar{x}}, v_{\bar{x}'})\}$, and coincides with the tail of $\bar{x}'$. So $\mathcal{B}^+$ is minimal.
\end{proof}
\subsection{Proof of Theorem \ref{thm:main}}
\label{subsec:proof1}
For the Bratteli diagram $\mathcal{B}$, let $w^\pm$ be weight functions on the positive and negative parts, respectively (Lemma \ref{lem:existence}). Recalling the conventions of \S \ref{subsec:renorm}, we let $w^\pm_k$ be the weight functions on the Bratteli diagram $\sigma^k(\mathcal{B})$ obtained by shifting those on $\mathcal{B}$.

Since the orbit of $\mathcal{B}$ has an accumulation point, there is a subsequence $k_i$ such that $\mathcal{B}_{k_i} \equiv \sigma^{k_i}(\mathcal{B}) \rightarrow \mathcal{B}^*$ for some $\mathcal{B}^* = (\mathcal{V}^*,\mathcal{E}^*)$. Let $\mathcal{F}^k_l$ be the $l^{th}$ transition matrix corresponding to the Bratteli diagram $\sigma^k(\mathcal{B})$ and $\mathcal{F}^*_l$ the $l^{th}$ matrix for the diagram $\mathcal{B}^*$. 

Denote $\Delta \equiv \Delta_{\mathcal{B}^*}^+(0)$ which is finite by Lemma \ref{lem:practicalTunnel}. Since the space $\beth$ is topologized by cylinder sets coming from the coding of $\mathfrak{M}$ into a countable alphabet, we can assume without loss of generality that $\mathcal{F}^{k_i}_l = \mathcal{F}^*_l$ for $-i \leq l \leq i$. Therefore, it follows that $\Delta^+_\mathcal{B}(k_i) = \Delta_{\mathcal{B}^*}^+(0)$ for all $i>\Delta$.

We will now find weights $w^\pm_*$ which make $(\mathcal{B}^*,w^\pm_*)$ compatible with the convergence $\sigma^{k_i}(\mathcal{B})\rightarrow \mathcal{B}^*$. Consider the countable collection of vectors
\begin{equation*}
\begin{split}
\mathbf{w}_j^i &= \left(w^+_{k_i}(v_1),\dots, w^+_{k_i}(v_{|\mathcal{V}_j^*|})\right) = \left(\bar{\ell}^{k_i}_{v_1},\dots,  \bar{\ell}^{k_i}_{v_{|\mathcal{V}_j^*|}} \right)  \in [0,1]^{|\mathcal{V}_j^*|} \\
\mathbf{h}_j^i &= \left(w^-_{k_i}(v_1),\dots, w^+_{k_i}(v_{|\mathcal{V}_{j}^*|})\right) = \left( \bar{h}^{k_i}_{v_1},\dots,  \bar{h}^{k_i}_{v_{|\mathcal{V}_{j}^*|}}\right) \in \mathbb{R}^{|\mathcal{V}_{j}^*|}_+
\end{split}
\end{equation*}
for $i>j\geq 0$. We can now find a weight function (invariant measure) for $\mathcal{B}^*$. By the convergence $\sigma^{k_i}(\mathcal{B})\rightarrow \mathcal{B}^*$, since $[0,1]^{|\mathcal{V}_0^*|}$ is compact there is a subsequence $k^0_i$ such that 
$$\mathbf{w}_0^{k^0_i} = \left(w^+_{k^0_i}(v_1),\dots, w^+_{k^0_i}(v_{|\mathcal{V}_0^*|})\right) \longrightarrow \mathbf{w}^*_0  = \left(w^+_*(v_1),\dots, w^+_*(v_{|\mathcal{V}_0^*|})\right)$$
 as $k_i^0\rightarrow \infty$.  We now proceed by recursion: suppose $\mathbf{w}_j^{k^n_i}\rightarrow \mathbf{w}^*_j$ along some subsequence $k^n_i \rightarrow \infty$ for all $0\leq j\leq n$ for some $n$. Since $[0,1]^{|\mathcal{V}_{n+1}^*|}$ is compact, there is a subsequence $k^{n+1}_i$ of $k^n_i$ such that $\mathbf{w}_j^{k^{n+1}_i}\rightarrow \mathbf{w}^*_j$ for all $0\leq j\leq n+1$. Doing this for all $n\in\mathbb{N}$, by choosing a diagonal subsequence $\kappa_n\rightarrow\infty$, we have that for any $v\in \mathcal{V}_k$, $w_{\kappa_n}^+(v)\rightarrow w^+_*(v)$ as $n\rightarrow \infty$, that is, we obtain a weight function function $w_*^+$ for (the positive part of) $\mathcal{B}^*$. 

Since $\mathbf{w}_0^{k_i^0}\cdot \mathbf{h}_0^{k_i^0}  = 1$ for all $k^0_i$ we have that $\sup \| \mathbf{h}_0^{k_i^0} \|_\infty < \infty$ and thus, without loss of generality (by passing to a further subsequence), we may assume that 
$$\mathbf{h}_0^{\kappa_n} = \left(w^-_{\kappa_n}(v_1),\dots, w^-_{\kappa_n}(v_{|\mathcal{V}_0^*|})\right) \longrightarrow \mathbf{h}^*_0  = \left(w^-_*(v_1),\dots, w^-_*(v_{|\mathcal{V}_0^*|})\right).$$
Again, since $\mathbf{w}_0^{k_i^0}\cdot \mathbf{h}_0^{k_i^0}  = 1$ for all $k^0_i$, we also have that $\mathbf{h}_0^{\kappa_n} \not\rightarrow \bar{0}$.
\subsection{Ergodicity when $\mathcal{B}^*$ has a minimal positive part}
Suppose that the accumulation point $\mathcal{B}^*$ has a minimal positive part. By minimality, it follows that $w^+_*$ is non-atomic. We claim that $\mathbf{w}^*_j \in \mathbb{R}_+^{|\mathcal{V}_j^*|}$ for all $j\geq 0$. Indeed, suppose that for some $\mathbf{w}^*_j = (\mathbf{w}^1_j,\dots, \mathbf{w}^{|\mathcal{V}^*_j|}_j)$ we have that $\mathbf{w}^i_j = 0$ for some $i$, i.e., $w^+_*(v_i) = 0$ for $v_i \in \mathcal{V}^*_j$. Then every tower corresponding to some $v'\in \mathcal{V}^*_k$ with $k>j$ which is connected to $v_i$ by some path must also satisfy $w^+_*(v') = 0$. By the invariance condition, all other paths ending in $r^{-1}(v')$ must also have zero weight. In particular, if there is a path from $v''\in \mathcal{V}_0$ to $v'$, then $w^+_*(v'') = 0$. It follows by minimality that $w^+_*(v) = 0$ for all $v\in \mathcal{V}_0$, a contradiction to $w^+$ being a probability measure.

We now want to apply Theorem \ref{thm:integrability} to the surface $S(\mathcal{B},w^\pm, \leq_{r,s})$, where $\leq_{r,s}$ is any ordering. Note that by Lemma \ref{lem:minimality} the positive part of $\mathcal{B}$ is minimal, which implies that the vertical flow on $S(\mathcal{B},w^\pm, \leq_{r,s})$ is minimal in the sense that any point on the surface whose orbits are defined for all time have a dense orbit (this is shown in \cite{LT}). We will see that the ordering will be irrelevant for ergodicity, and thus the proof is independent of ordering. Let $\kappa_n\rightarrow \infty$ be the sequence chosen so that $w^+(v)\rightarrow w_*^+(v)$, and let $t_{\kappa_n}\rightarrow \infty$ be the sequence of times defined by the subsequence $\kappa_n$ and (\ref{eqn:rTimes}). Without loss of generality, we can also assume that $\inf (t_{\kappa_{n+1}} - t_{\kappa_n}) > 0$. By Proposition \ref{prop:shift}, $S(\sigma^{\kappa_n}(\mathcal{B},w^\pm,\leq_{r,s}))$ is isometric to $g_{t_{\kappa_n}}S(\mathcal{B},w^\pm,\leq_{r,s})$. Let $\eta>0$ be given.

There is a partition $\mathcal{V}_0^* = \mathcal{G}_0 \sqcup \mathcal{H}_0$, where the vertices $v\in \mathcal{G}_0$ correspond to vertices $v\in \sigma^{\kappa_n}(\mathcal{V}_{\kappa_n}^*)$ with the property that $w_{\kappa_n}^-(v)\not\rightarrow 0$ and the vertices $v\in \mathcal{H}_0$ correspond to vertices $v\in \sigma^{\kappa_n}(\mathcal{V}_{\kappa_n}^*)$ with the property that $w_{\kappa_n}^-(v)\rightarrow 0$. Note that $\mathcal{G}_0 \neq \varnothing$. 

Let $n$ be large enough so that $n>\Delta$ and that
\begin{equation}
\label{eqn:weightConv}
\max_{v\in \sigma^{\kappa_n+j}(\mathcal{V}_{\kappa_n+j})}\left|   w_*^+(v)  - w_{\kappa_n}^+(v) \right| \leq 10^{-2} \min_{v\in \mathcal{V}_j^*}w^+_*(v),
\end{equation}
for all $j\in\{0,\dots, \Delta\}$, where we identify $v\in\sigma^{\kappa_n+j}(\mathcal{V}_{\kappa_n+j})$ with its corresponding vertex $v\in\mathcal{V}_j^*$ given by the convergence. Recall that the surface $g_{t_{\kappa_n}}S(\mathcal{B},w^\pm,\leq_{r,s})$ is isometric to the surface $S(\sigma^{\kappa_n}(\mathcal{B},w^\pm,\leq_{r,s}))$ and that by (\ref{eqn:RecIds})
\begin{equation}
\label{eqn:moreRect}
S(\sigma^{\kappa_n}(\mathcal{B},w^\pm,\leq_{r,s})) = \left(\bigcup_{i=1}^{|\mathcal{V}_0^*|} R_i^n \backslash \Sigma   \right) / \sim_n ,
\end{equation}
where, by (\ref{eqn:rectangles}),
$$R_i^n = \left[\sum_{i<j}w^+_{\kappa_n}(v_j),\sum_{i\leq j}w^+_{\kappa_n}(v_j)\right]\times \left[\sum_{i<j}w^-_{\kappa_n}(v_j),\sum_{i\leq j}w^-_{\kappa_n}(v_j)\right]$$
and $\sim_n$ is an identification on $\bigcup_i\partial R_i^n$. For any $n$ large enough so that (\ref{eqn:weightConv}) holds and any $\varepsilon\in (0, \min_{v\in\mathcal{V}_0^*}w^+_*(v)/3)$, consider the set
\begin{equation}
\label{eqn:GoodSets}
\begin{split}
S_n^{\varepsilon}(\mathcal{B},w^\pm, \leq_{r,s}) &= S^{\varepsilon}(\sigma^{\kappa_n}(\mathcal{B},w^\pm, \leq_{r,s})) \\
&:= \left\{ z\in S(\sigma^{\kappa_n}(\mathcal{B},w^\pm,\leq_{r,s})): \mbox{dist}_{t_{\kappa_n}}(z,\bigcup_{i\in \mathcal{G}_0} \partial R_i^n) \geq \varepsilon\right\} \\
&\subset S(\sigma^{\kappa_n}(\mathcal{B},w^\pm, \leq_{r,s})) = g_{t_{\kappa_n}}S(\mathcal{B},w^\pm, \leq_{r,s}),
\end{split}
\end{equation}
where $\mbox{dist}_{t_{\kappa_n}}(\cdot, \cdot)$ is the flat metric on $S(\sigma^{\kappa_n}(\mathcal{B},w^\pm, \leq_{r,s})) = g_{t_{\kappa_n}}S(\mathcal{B},w^\pm, \leq_{r,s})$ and we view $S(\sigma^{\kappa_n}(\mathcal{B},w^\pm,\leq_{r,s}))$ as a collection of rectangles with identifications on the edges as in (\ref{eqn:moreRect}). Note that as defined above, we can express the set as
$$S_n^{\varepsilon}(\mathcal{B},w^\pm, \leq_{r,s}) = \bigsqcup_{i=1}^{|\mathcal{G}_0|} S_{\varepsilon,n}^i.$$
For $\eta$ small enough and $n$ large enough, let $\varepsilon(\eta,n)$ be any number so that we have that  $\mathrm{Area}(S_n^{\varepsilon}(\mathcal{B},w^\pm, \leq_{r,s})) \geq 1- \eta$ whenever $\varepsilon\in (0,\varepsilon(\eta,n))$. This is possible since the rectangles indexed by vertices in $\mathcal{H}_0$ are degenerate, i.e., $\mathrm{Area}(R_v^{\kappa_n})\rightarrow 0$ for any $v\in\mathcal{H}_0$. To apply Theorem \ref{thm:integrability} we will first pick sets for the sequence of times $t_{\kappa_n}\rightarrow \infty$:
\begin{equation}
\label{eqn:setChoice}
S_{\varepsilon(t_{\kappa_n}),t_{\kappa_n}} =S_n^{\varepsilon}(\mathcal{B},w^\pm, \leq_{r,s}) = \bigsqcup_{i=1}^{|\mathcal{G}_0|} S_{\varepsilon,n}^i,
\end{equation}
where $\varepsilon < \varepsilon(\eta,n)$.
\begin{lemma}
\label{lem:pathDist}
Given $n>\Delta$ large enough so that (\ref{eqn:weightConv}) holds, there exists a $\delta_{\varepsilon,\Delta}$ such that for any $z_i \in S^i_{\varepsilon, n}, z_j \in S^j_{\varepsilon, n}$ with $i\neq j \in \mathcal{G}_0$ there is a path $\gamma:[0,1]\rightarrow g_{t_{\kappa_n}}S(\mathcal{B},w^\pm, \leq_{r,s})$ with $\gamma(0) = z_i$, $\gamma(1) = z_j$ and 
\begin{equation}
\label{eqn:pathDist}
\min_{\tau\in[0,1]}\mbox{dist}_{t_{\kappa_n}}(\gamma(\tau), \Sigma) \geq \delta_{\varepsilon,\Delta}.
\end{equation}
\end{lemma}
\begin{proof}
Recall that by Proposition \ref{prop:shift} and (\ref{eqn:moreRect}) the surface $g_{t_{\kappa_n}}S(\mathcal{B},w^\pm, \leq_{r,s})$ can be defined as the union of $|\mathcal{V}_0^*|$ rectangles along with identifications on the edges of the rectangles, and that we are considering the subcollection of non-degenerate rectangles indexed by the vertices in $\mathcal{G}_0$. Let $v_i, v_j\in \sigma^{\kappa_n}(\mathcal{V}_{\kappa_n})$ represent the rectangles $R_i^n$ and $R_j^n$ for $i,j\in\mathcal{G}_0$, respectively. By the convergence $\sigma^{\kappa_n}(\mathcal{B})\rightarrow \mathcal{B}^*$ and Lemma \ref{lem:practicalTunnel} we have that $\Delta_{\mathcal{B}}^+(\kappa_n) = \Delta < \infty$ for all $n$ large enough. This means that there exist paths $p, q \in \mathcal{E}_{\kappa_n,\kappa_n + \Delta} $ with $s(p) = v, s(q) = w$ and $r(p) = r(q)$.

The surface $S(\sigma^{\kappa_n+\Delta}(\mathcal{B},w^\pm, \leq_{r,s})) = g_{t_{\kappa_n+\Delta}}S(\mathcal{B},w^\pm, \leq_{r,s})$ is obtained from $S(\sigma^{\kappa_n}(\mathcal{B},w^\pm, \leq_{r,s}))$ by deforming $S(\sigma^{\kappa_n}(\mathcal{B},w^\pm, \leq_{r,s}))$ by $g_{t_{\kappa_n + \Delta} - t_{\kappa_n}}$, cutting of each $R^{\kappa_n}_v$ into rectangles and stacking them into the rectangles $R^{\kappa_n+\Delta}_{v'}$, $v'\in \sigma^{\kappa_n}(\mathcal{V}_{\kappa_n+\Delta})$. As such, a path $p\in\mathcal{E}_{\kappa_n,\kappa_n+\Delta}$ guaranteed by the finiteness of the metamour function represents a class of paths going from rectangle $R^{\kappa_n}_{s(p)}$ to $R^{\kappa_n+\Delta}_{r(p)}$. The same holds for the path $q$. Since the width of $R^{\kappa_n +\Delta}_{s(p)}$ with respect to the flat metric on $S(\sigma^{\kappa_n}(\mathcal{B},w^\pm, \leq_{r,s})) = g_{t_{\kappa_n}}S(\mathcal{B},w^\pm, \leq_{r,s})$ is no less than
$$\min_{v\in \sigma^{\kappa_n}(\mathcal{V}_{\kappa_n+\Delta})}w^+_{\kappa_n}(v)$$
and we picked $n$ large enough so that (\ref{eqn:weightConv}) holds, there is a path $\gamma$ from $R^{\kappa_n}_{s(p)}$ to $R^{\kappa_n+\Delta}_{s(q)}$ which satisfies 
$$\mbox{dist}_{t_{\kappa_n}}(\gamma, \Sigma) \geq \min \left\{ \varepsilon, \min_{v\in \sigma^{\kappa_n}(\mathcal{V}_{\kappa_n+\Delta})} \frac{w^+_{\kappa_n}(v)}{3}  \right\}.$$
Setting
$$\delta_{\varepsilon,\Delta} := \min \left\{ \varepsilon, \min_{v\in \mathcal{V}_{\Delta}^*} \frac{w^+_{*}(v)}{4}  \right\},$$
the result follows by (\ref{eqn:weightConv}).
\end{proof}
For the sequence of times $t_{\kappa_n}\rightarrow\infty$, we choose $\varepsilon(t_{\kappa_n}) = \varepsilon$ and sets $S_{\varepsilon(t_{\kappa_n}),t_{\kappa_n}}$ by (\ref{eqn:setChoice}). We can readily see that we can bound the diameter of each rectangle $R^{\kappa_n}_v$ by $4(\|\mathbf{w}_0^*\|_\infty+ \|\mathbf{h}_0^*\|_\infty)$ with respect to the flat metric on $g_{t_{\kappa_n}}S(\mathcal{B},w^\pm, \leq_{r,s})$, so we choose $\mathcal{D}_{t_{\kappa_n}}^i = 4(\|\mathbf{w}_0^*\|_\infty+ \|\mathbf{h}_0^*\|_\infty)$ for all $i$ and $\delta_{t_{\kappa_n}} = \delta_{\varepsilon,\Delta}$ for all $n$ large enough.

With these choices of $\eta,\varepsilon(t_{\kappa_n}),\mathcal{D}_{t_{\kappa_n}}^i,$ and, by Lemma \ref{lem:pathDist}, $ \delta_{t_{\kappa_n}}$ we satisfy the hypotheses (i)-(iii) in Theorem \ref{thm:integrability}. Moreover, since they are constant,
\begin{equation}
\label{eqn:discreteDiv}
\sum_{n>0} \left( \varepsilon(t_{\kappa_n})^{-2}\sum_{i=1}^{|\mathcal{G}_0|}\mathcal{D}_{t_{\kappa_n}}^i + \frac{|\mathcal{G}_0|-1}{\delta_{t_{\kappa_n}}}\right)^{-2} =   
\sum_{n>0} \left( \frac{4|\mathcal{G}_0|(\|\mathbf{w}_0^*\|_\infty+ \|\mathbf{h}_0^*\|_\infty)}{\varepsilon^2} + \frac{|\mathcal{G}_0|-1}{\delta_{\varepsilon,\Delta}}\right)^{-2} = \infty.
\end{equation}
Let $\mu > 0$ be such that $2\mu \ll \inf(t_{\kappa_{n+1}})$. Using the Teichm\"uller deformation, we deform the sets $S^\varepsilon_n(\mathcal{B},w^\pm,\leq_{r,s})$ for time intervals of width $\mu/2$ around each $t_{\kappa_n}$. As such, the geometric properties of the sets are closely related to those of the sets $S^\varepsilon_n(\mathcal{B},w^\pm,\leq_{r,s})$. In particular,
for $t\in (t_{\kappa_n} - \frac{\mu}{2}, t_{\kappa_n} + \frac{\mu}{2})$, if we define
\begin{equation}
\label{eqn:intervalBnds}
\mathcal{D}_t^i = e^{\frac{\mu}{2}}\mathcal{D}_{t_{\kappa_n}}^i \mbox{ for all }i,\,\, \varepsilon(t) = e^{-\frac{\mu}{2}}\varepsilon,\,\, \mbox{ and } \delta_t = e^{-\frac{\mu}{2}}\delta_{\varepsilon,\Delta}
\end{equation}
we satisfy conditions (i)-(iii) in Theorem \ref{thm:integrability} for $t$ in these intervals. Making any choice of sets $S_{\varepsilon(t),t}$ for $t\in (\mathbb{R}\backslash \bigcup_n(t-\frac{\mu}{2},t+\frac{\mu}{2}))$ we have that
\begin{equation}
\label{eqn:bigDivergence}
\begin{split}
\int_0^\infty \left( \varepsilon(t)^{-2}\sum_{i=1}^{C_t}\mathcal{D}_t^i + \frac{C_t-1}{\delta_t}\right)^{-2}\, dt &\geq \sum_{n>0} \int_{t_{\kappa_n}-\frac{\mu}{2}}^{t_{\kappa_n}+\frac{\mu}{2}} \left( \varepsilon(t)^{-2}\sum_{i=1}^{C_t}\mathcal{D}_t^i + \frac{C_t-1}{\delta_t}\right)^{-2}\, dt = +\infty \\
&\geq \mu\sum_{n>0}\left(  e^{\frac{3\mu}{2}} \varepsilon^{-2}  \sum_{i=1}^{|\mathcal{G}_0|}\mathcal{D}_{t_{\kappa_n}}^i + e^{\frac{\mu}{2}}\frac{|\mathcal{G}_0|-1}{\delta_{t_{\kappa_n}}}     \right)^{-2}\\
&\geq \mu e^{3\mu} \sum_{n>0} \left( \frac{4|\mathcal{G}_0|(\|\mathbf{w}_0^*\|_\infty+ \|\mathbf{h}_0^*\|_\infty)}{\varepsilon^2} + \frac{|\mathcal{G}_0|-1}{\delta_{\varepsilon,\Delta}}\right)^{-2} \\
&= \infty,
\end{split}
\end{equation}
where we used (\ref{eqn:discreteDiv}). Therefore, by Theorem \ref{thm:integrability}, the vertical flow on $S(\mathcal{B},w^\pm,\leq_{r,s})$ is ergodic.
\subsection{Upgrade}
\label{subsec:upgrade}
In this section we upgrade the ergodicity result obtained in the previous section to unique ergodicity, that is, we will show that the Lebesgue measure $\omega_\alpha = \omega_{(\mathcal{B},w^\pm,\leq_{r,s})}$ is the only invariant probability measure for the vertical flow. Let us suppose that $\mu$ is another invariant measure. Since the flow is minimal, $\mu$ is not atomic. The measures $\omega_\alpha$ and $\mu$ are mutually singular. Define the one parameter family of measures $\mu_s = s\omega_\alpha + (1-s)\mu$ for $s\in [0,1]$. All of these measures are invariant under the vertical flow on the surface $S(\mathcal{B},w^\pm,\leq_{r,s})$. We will show that $\mu_s$ is an ergodic measure for some $s\in(0,1)$, contradicting that ergodic measures cannot be convex combinations of other ergodic measures.

Let $s\in (0,1)$ be fixed. For $p\in S\backslash\Sigma$ and $\lambda>0$, define $L_{\lambda}^\alpha(p)$ to be a connected subset of the leaf of the horizontal foliation of $\alpha$ of length $\lambda$ (with respect to the flat metric) with $p$ at its ``left'' point. In other words,
$$L^\alpha_{\lambda}(p) = \bigcup_{\tau\in[0,\lambda]}\varphi^X_\tau(p).$$
Let $R_{\beta,\lambda}^\alpha(p)$ be the rectangle with $p$ at its lower-left corner with vertical and horizontal side-lengths $\lambda$ and $\beta$, respectively, with respect to the flat metric. More precisely,
$$R_{\beta,\lambda}^\alpha(p) = \bigcup_{\substack{r\in[0,\beta]\\ \tau\in[0,\lambda]}} \varphi^Y_r\circ\varphi^X_\tau(p).$$
Without loss of generality we implicitly assume that $\beta$ and $\lambda$ are small enough such that $R_{\beta,\lambda}(p)\cap\Sigma = \varnothing$.

Note that any invariant measure $\mu_*$ for the translation flow is locally a product measure $\mu_* = m_x\times m_y$ of measures $m_x$ on the horizontal leaves and $m_y$ on the vertical leaves of the foliation. Since the measure a small rectangle $R^\alpha_{\beta,\lambda}$ as above is invariant under the vertical flow (which locally is a translation) the measure in the flow direction is always the Lebesgue measure: $m_y = \mathrm{Leb}_y$.

Let $\mu$ be the other ergodic invariant probability measure for the vertical flow. By the remark above, we have that $\mu$ is locally of the form $\mu =   \Upsilon\times \mathrm{Leb}_y$, where $\Upsilon$ is transverse measure to the vertical flow on the horizontal leaves. As such, the convex combination $\mu_s$ is locally of the form $ \Upsilon_s \times \mathrm{Leb}_y$ for some measure $\Upsilon_s$ on the horizontal leaves. The measure of horizontal segments is then
$$\Upsilon_s(L_{\lambda}^\alpha(p)) = \lim_{\beta\rightarrow0}\frac{1}{\beta}\mu_s(R_{\beta,\lambda}^\alpha(p)).$$
Since
$$\mu_s(R_{\beta,\lambda}^\alpha(p)) = s\beta\lambda + (1-s)\mu(R_{\beta,\lambda}(p))\geq s\beta\lambda$$
it follows that
\begin{equation}
\label{eqn:vertLength}
\Upsilon_s(L_\lambda^\alpha(p)) = s\lambda + \lim_{\beta\rightarrow 0}\frac{1}{\beta}(1-s)\mu(R_{\beta,\lambda}(p))\geq s\lambda.
\end{equation}
As such, (\ref{eqn:vertLength}) gives a lower bound to transverse lengths of curves with respect to the measure $\Upsilon_s$. 

Construct the homeomorphism $\Phi_s: S\rightarrow S$ as follows. For a point in $z\in S\backslash \Sigma$ and using $w = x+iy$ as a local coordinate (identifying $z$ with zero in these local coordinates), the map $\Phi_s$ takes the local form 
$$\Phi_s:w\mapsto \Phi_s(w) = sign(x)\cdot  \Upsilon_s(L_{x}^\alpha(z)) + iy.$$ 
The map $\Phi_s$ induces a map on the measure $\mu_s$ which sends it to $\mu^s_*\equiv (\Phi_s)_*\mu_s $, the Lebesgue measure in $\Phi_s(S)$. Moreover, the map induces a new translation structure on $S$ and preserves the smooth foliation $\mathcal{F}_\alpha^v$. Considering the transverse foliations $\mathcal{F}^h_\alpha$ and $\mathcal{F}^v_\alpha$ with respect to this new translation structure obtained through the map $\Phi_s$, there exists an Abelian differential $\alpha_s$ such that $\mathcal{F}^{h,v}_\alpha = \mathcal{F}^{h,v}_{\alpha_s}$. By construction, $\mu_*^s = \omega_{\alpha_s}$. We denote this surface as $(S_\mathcal{B},\alpha_s)$.

We now compare the geometry of the sets (\ref{eqn:GoodSets}) under the flat metric given by $\alpha_s$ to the geometry of the sets given by the metric corresponding to $\alpha$.

\begin{lemma}
\label{lem:2areas}
For any $\eta > 0$ there exists a $\varepsilon>0$ such that the sets $S^\varepsilon_{\kappa_n}$ defined by (\ref{eqn:GoodSets}) and the subsequence $\kappa_n\rightarrow \infty$ through which the geometry of the surfaces $S(\sigma^{\kappa_n}(\mathcal{B},w^\pm,\leq_{r,s}))$ converges, satisfy
$\omega_\alpha(S^\varepsilon_n) \geq 1-\eta\mbox{ and }\mu^s_*(S^\varepsilon_{\kappa_n}) \geq 1-\eta$ for all $n$.
\end{lemma}
\begin{proof}
By (\ref{eqn:vertLength}), the fact that the sets $S^\varepsilon_n$ are an exhaustion of $S(\sigma^n(\mathcal{B},w^\pm,\leq_{r,s}))$, and that the geometry of the sequence of surfaces $S(\sigma^{\kappa_n}(\mathcal{B},w^\pm,\leq_{r,s}))$ converges, the result follows.
\end{proof}

We now apply Theorem \ref{thm:integrability} to the surface $(S_\mathcal{B},\alpha_s)$. Let $\eta>0$ and choose $\varepsilon >0$ so that the conclusions of Lemma \ref{lem:2areas} hold. Consider the sequences $\kappa_n\rightarrow\infty$ and $t_{\kappa_n}\rightarrow\infty$ obtained in \S \ref{subsec:proof1}. We consider the sets $S^\varepsilon_{\kappa_n}$ as subsets of $g_{t_{\kappa_n}}(S_\mathcal{B},\alpha_s)$. By (\ref{eqn:vertLength}), we can now measure the geometry of these subsets using the geometry of $g_{t_{\kappa_n}}S(\mathcal{B},w^\pm,\leq_{r,s})$. Note that
\begin{enumerate}
\item For $n$ large enough, a subset of the surface $g_{t_{\kappa_n}}(S_\mathcal{B},\alpha_s)$ of area arbitrarily close to 1 can be represented as $|\mathcal{G}_0|$ rectangles (whose heights are converging to the non-zero entries of $\mathbf{h}_0^*$ since we have not changed the vertical lengths) with identifications given on the boundaries of the rectangles;
\item By Proposition \ref{prop:shift} and (\ref{eqn:vertLength}), for $n$ large enough, the widths of each of these $|\mathcal{G}_0|$ rectangles is bounded below by $\frac{s}{2}$ times the width of the rectangles corresponding to the surface $g_{t_{\kappa_n}}S(\mathcal{B},w^\pm,\leq_{r,s})$ and, since the area of the surface is 1 and the heights of rectangles are bounded, the widths of the rectangles are also bounded. Thus, the diameter of $S^\varepsilon_{\kappa_n}$ with respect to the flat metric on $g_{t_{\kappa_n}}(S_\mathcal{B},\alpha_s)$ is bounded above by $s^{-1}$ times the diameter of each rectangle with respect to the flat metric on $g_{t_{\kappa_n}}S(\mathcal{B},w^\pm,\leq_{r,s})$;
\item By (\ref{eqn:vertLength}) and Lemma \ref{lem:pathDist}, for any $z_i$ and $z_j$ in different rectangles of $g_{t_{\kappa_n}}(S_\mathcal{B},\alpha_s)$, there exists a path $\gamma:[0,1]\rightarrow g_{t_{\kappa_n}}(S_\mathcal{B},\alpha_s)$ with $\gamma(0) = z_i$ and $\gamma(1) = z_j$ and 
$$ \min_{\tau\in[0,1]}\mbox{dist}_{\kappa_n, s}(\gamma(\tau), \Sigma) \geq s\cdot \delta_{\varepsilon,\Delta}, $$
where $\mbox{dist}_{\kappa_n, s}$ is the distance with respect to the flat metric on $g_{t_{\kappa_n}}(S_\mathcal{B},\alpha_s)$ and $\delta_{\varepsilon,\Delta}$ is the quantity obtained from Lemma \ref{lem:pathDist} for the surface $g_{t_{\kappa_n}}S(\mathcal{B},w^\pm,\leq_{r,s})$.
\end{enumerate}
As such, for the surface $(S_\mathcal{B},\alpha_s)$ we choose $\mathcal{D}_{t_{\kappa_n}}^i = \frac{4(\|\mathbf{w}_0^*\|_\infty+ \|\mathbf{h}_0^*\|_\infty)}{s}$, $\delta_{t_{\kappa_n}} = s\cdot \delta_{\varepsilon,\Delta}$ and $\varepsilon(t_{\kappa_n}) = s\cdot \varepsilon$. We can now apply Theorem \ref{thm:integrability}: with the chosen geometric quantities for the sequence of times $t_{\kappa_n}$ we get that
$$\sum_{n>0} \left( \varepsilon(t_{\kappa_n})^{-2}\sum_{i=1}^{|\mathcal{G}_0|}\mathcal{D}_{t_{\kappa_n}}^i + \frac{|\mathcal{G}_0|-1}{\delta_{t_{\kappa_n}}}\right)^{-2} =   
\sum_{n>0} \left( \frac{4|\mathcal{G}_0|(\|\mathbf{w}_0^*\|_\infty+ \|\mathbf{h}_0^*\|_\infty)}{s^3\varepsilon^2} + \frac{|\mathcal{G}_0|-1}{s\delta_{\varepsilon,\Delta}}\right)^{-2} = \infty.$$
By the same argument which uses (\ref{eqn:intervalBnds}) and (\ref{eqn:bigDivergence}), we see that (\ref{eqn:integrability2}) holds for $(S_\mathcal{B},\alpha_s)$, making the vertical flow ergodic. But this contradicts that the Lebesgue measure $\omega_{\alpha_s}$ on $(S_\mathcal{B},\alpha_s)$ is a convex combination of ergodic measures. Therefore, the Lebesgue measure is the unique probability measure on $S(\mathcal{B},w^\pm,\leq_{r,s})$ which is invariant by the vertical flow.
\subsection{An example}
\label{subsec:example}
In this section we construct a family of examples which show the fact that the orbit of a Bratteli diagram $\mathcal{B}$ can have an accumulation point $\sigma^{n_k}(\mathcal{B})\rightarrow \mathcal{B}^*$ with the positive part of $\mathcal{B}^*$ being transitive and admiting a unique invariant non-atomic probability measure but $\mathcal{B}$ does not admit \textbf{any} non-atomic invariant probability measure.

For natural numbers $n,p$, let 
$$ M(p,n) = \left( \begin{array}{cc}
p&n \\ 0&1 \end{array}\right). $$
Let $\mathcal{B}^*$ be a Bratteli diagram $\mathcal{B}^*$ with transition matrix $\mathcal{F}_k = M(3,1)$ for all $k\in\mathbb{Z}\backslash \{0\}$. The positive part of $\mathcal{B}^*$, under the right orderings, corresponds to Chacon's transformation (see \cite[\S 6.2]{LT}), which is known to be uniquely ergodic. The positive part of $\mathcal{B}^*$ does not admit a unique invariant probability weight function because there is an atomic weight function supported on the ``rightmost path'' of the positive part of the Bratteli diagram $\mathcal{B}^*$. But there is a unique \textbf{non-atomic} probability weight function defined on the positive part of $\mathcal{B}^*$, which reflects the fact that the Chacon transformation is uniquely ergodic.

Fix $p\in\mathbb{N}$ and let $\bar{n} = (n_1,n_2,n_3,\dots)\in \mathbb{N}^\mathbb{N}$. We will now define a Bratteli diagram $\mathcal{B}(\bar{n})$ by specifying the transition matrices $\mathcal{F}_k$, $k\in\mathbb{Z}\backslash \{0\}$. For $k<0$, $\mathcal{F}_k = \mathrm{Id}_2$. For $k\in\mathbb{N}$, 
$$\mathcal{F}_k = \left\{ \begin{array}{ll} 
M(p,n_i)&\mbox{ if }k = (i+1)^2-1 \\ M(p,1)&\mbox{ otherwise.} \end{array}\right.$$
With $p=3$, for any $\bar{n}\in\mathbb{N}^\mathbb{N}$, we have that along the subsequence $k_i = i(i+1)$, $i\in\mathbb{N}$, we have that $\sigma^{k_i}(\mathcal{B})\rightarrow \mathcal{B}^*$. Moreover, we have that the positive part of $\mathcal{B}(\bar{n})$ is transitive: in fact, all but one infinite path belongs to a tail-equivalence class which is dense in $X_{\mathcal{B}^+}$. The exceptional path is the ``rightmost'' path running down the right side of the diagram. This path forms a periodic component containing one element. As such, it supports an atomic probability weight function.

Suppose that $w^+$ is a non-atomic weight function for the positive part of $\mathcal{B}(\bar{n})$ and normalize it so that $w^+(v_1) = 1$, where $v_1\in\mathcal{V}(\bar{n})_0$ is the vertex with $p$ edges coming out of it. For any $k\in\mathbb{N}$, let $v_1^k\in\mathcal{V}(\bar{n})_k$ be the vertex with $p$ edges emanating from it. Then it is straight forward to work out that any weight function $w^+$ must satisfy
$$w^+(v_1^k) = \sum_{i=0}^k  \frac{a_i}{p^i},$$
where $a_i = n_i$ if $i = (j+1)^2-1$ for some $j\in\mathbb{N}$ and otherwise $a_i = 1$. Then it follows that if the $n_i$ grow fast enough, the weight function $w^+$ will not be a finite weight function (it will be forced to have $w^+(v_2) = \infty$, for $v_2\in\mathcal{V}(\bar{n})_0$ which is not $v_1$). This, for example, holds for $n_i = p^{(i+1)^2-1}$. Thus there are many Bratteli diagrams $\mathcal{B}$ such that they have an accumulation point along a subsequence of the shift map in $\mathcal{B}^*$, the Bratteli diagram corresponding to the (uniquely ergodic) Chacon transformation, and which do not admit a finite, non-atomic invariant probability measure.

We can still deduce that any adic transformation defined on the positive part of $\mathcal{B}(\bar{n})$ is ergodic, but somewhat indirectly: it suffices to note that $\mathcal{B}(\bar{n})$ it is a tower over the $p$-adic odometer, which is uniquely ergodic (see also \cite[\S 8.2.5]{LT}).

\appendix
\section{Fibonacci Bratteli diagrams}
\label{sec:fib}

\begin{figure}[t]
  \centering
  \includegraphics[width=0.8\textwidth]{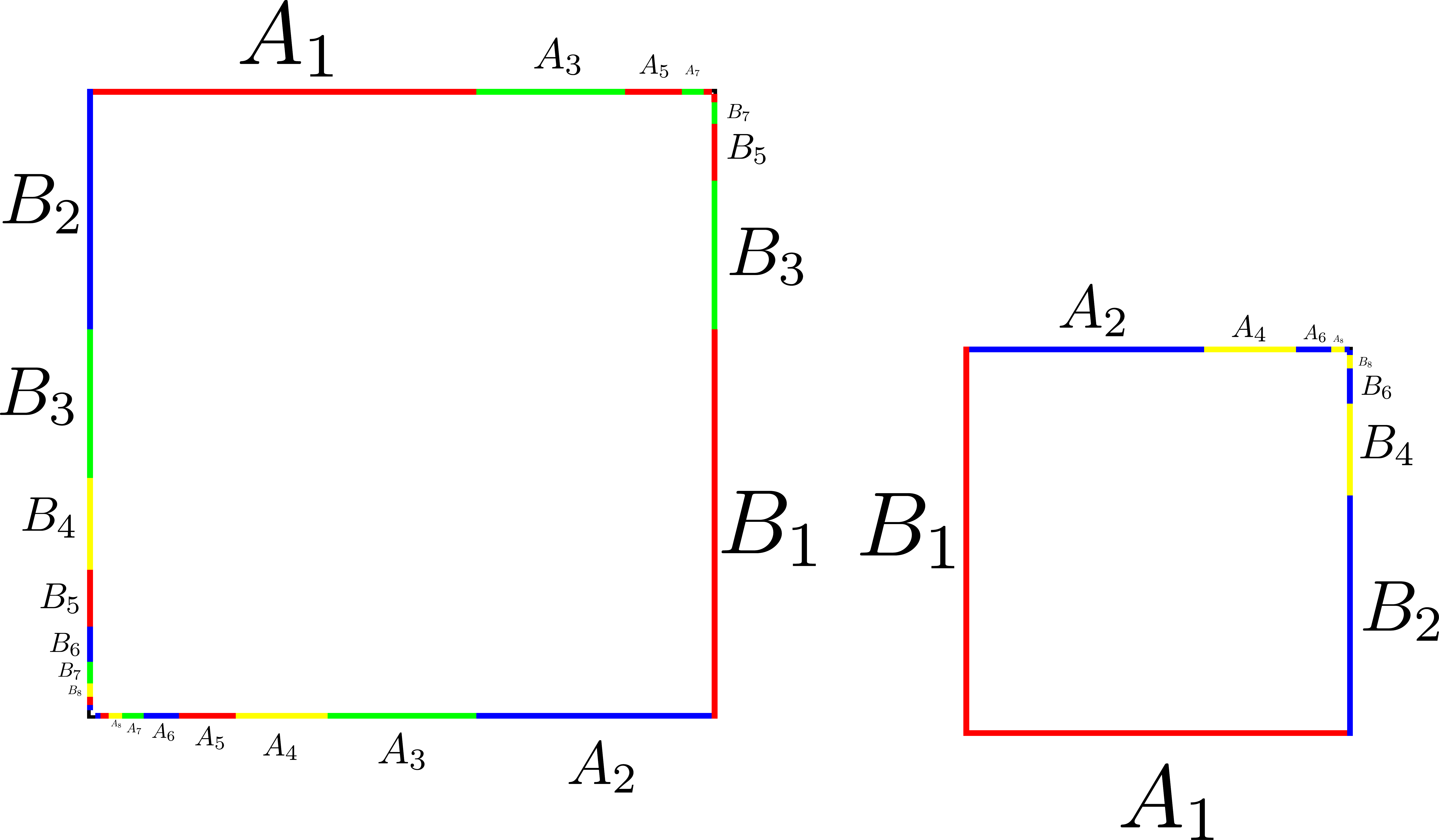}
  \caption{The flat surface associated to the stationary Bratteli diagram given by the Fibonacci substitution. The interiors of the edges with the same $A_i$ label are identified and so are the ones with the same $B_i$ labels.}
  \label{fig:fib}
\end{figure}

In this section I will work out an explicit construction of a surface from a particular Bratteli diagram. The hope is that this example will be illuminating and help the reader understand the different definitions and how they fit together.

When drawing surfaces from Bratteli diagrams, the easiest ones to approach are diagrams which are stationary (see Definition \ref{def:stationary}). In fact, by a procedure known as telescoping, for the purposes here, any stationary bi-infinite Bratteli is equivalent to one which is given by a single matrix. In the example here I will consider the stationary Bratteli diagram which is given by the single matrix
$$ \mathfrak{F} = \left( \begin{array}{cc}
1&1 \\ 1&0 \end{array}\right),$$
which is the matrix corresponding to the so-called Fibonacci substitution. The bi-infinite Bratteli diagran $\mathcal{B}$ is defined by the transition matrices $\mathcal{F}_k = \mathfrak{F}$ for all $k\in \mathbb{Z}\backslash \{0\}$. Since $|\mathcal{V}_k| = 2$ for all $k\in\mathbb{Z}$ we denote by $v_{i}^k$, $k\in\mathbb{Z}$, $i\in \{1,2\}$ the $i^{th}$ vertex in $\mathcal{V}_k$.

The reason why depicting surfaces from Bratteli diagrams is significantly easier when the diagram is stationary is that the weight function $w$ comes from a Perron-Frobenius eigenvector $v$ of the single matrix defining the matrix. For $\mathfrak{F}$ it is $v = (\frac{1+\sqrt{5}}{2} , 1)$, and so we have for this weight function that $w^\pm(v_1^0) = \frac{2}{1+\sqrt{5}}$ and $w^\pm(v_2^0) = \frac{1}{3+\sqrt{5}}$. Moreover, for edges coming out of $v_1^k$ their corresponding weights are $\frac{\sqrt{5}-1}{2}$ and $\frac{3-\sqrt{5}}{2}$, respectively, for all $k$.

\begin{wrapfigure}{l}{0.25\textwidth}
  \centering
  \includegraphics[width=0.25\textwidth]{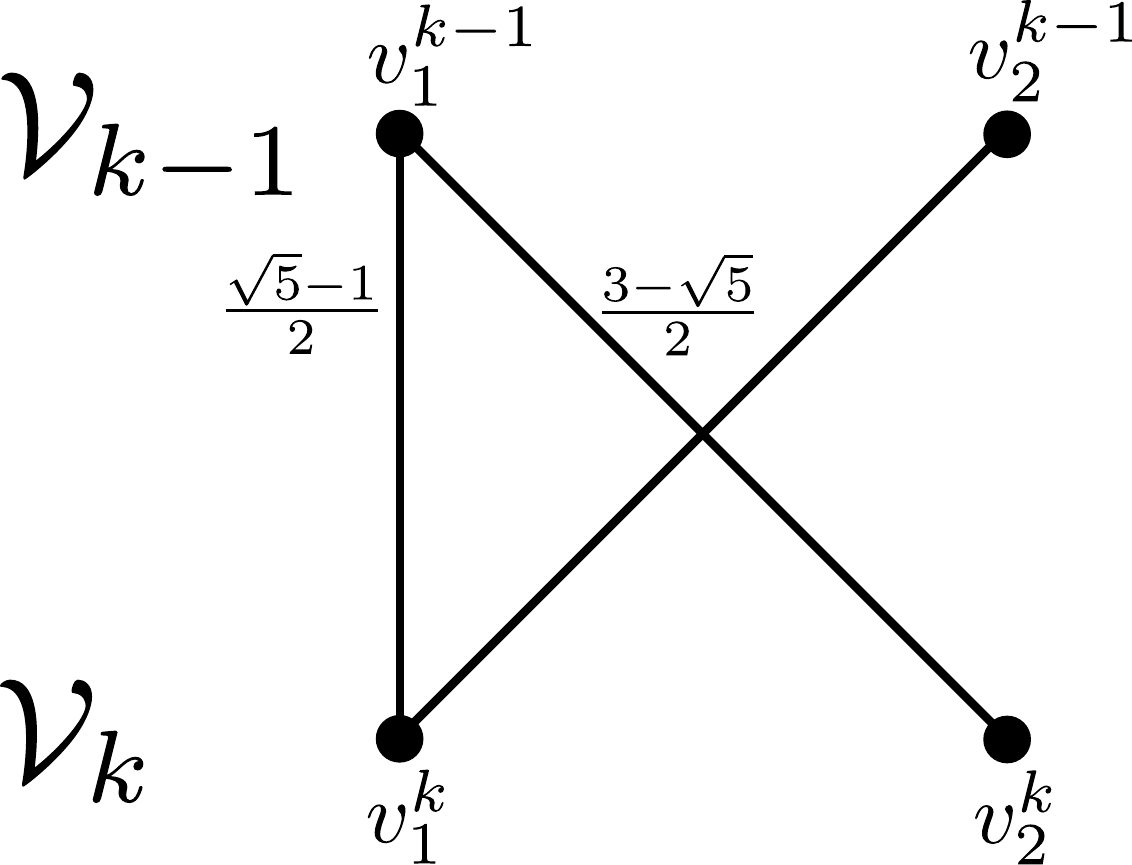}
\end{wrapfigure}

We can now move to the orderings. The most ``natural'' ordering is the one which one can read from the depiction of $\mathfrak{F}$ as a level of the Bratteli diagram (see the image on the left). That is, for $\leq_r$, for every $k\in\mathbb{Z}$, we order the two edges going into $v_1^k$ as they appear in the figure on the left, that is, from left to right. For $\leq_s$, for any $k\in\mathbb{Z}$, we order the edges coming out of $v_1^{k-1}$ as they come out as depicted on the image on the left, that is, from left to right. So now we have defined our orderings $\leq_{r,s}$.

Now that the weight functions $w^\pm$ and orderings $\leq_{r,s}$ are in place, we can construct the flat surface $S(\mathcal{B},w^\pm, \leq_{r,s})$. Recall that to do this we first need to define an interval exchange transformation (on infinitely many intervals) through the process of cutting and stacking (\cite[\S 6]{LT}). Figure \ref{fig:CAS} depicts the process described below.

\begin{figure}[t]
  \centering
  \includegraphics[width=0.85\textwidth]{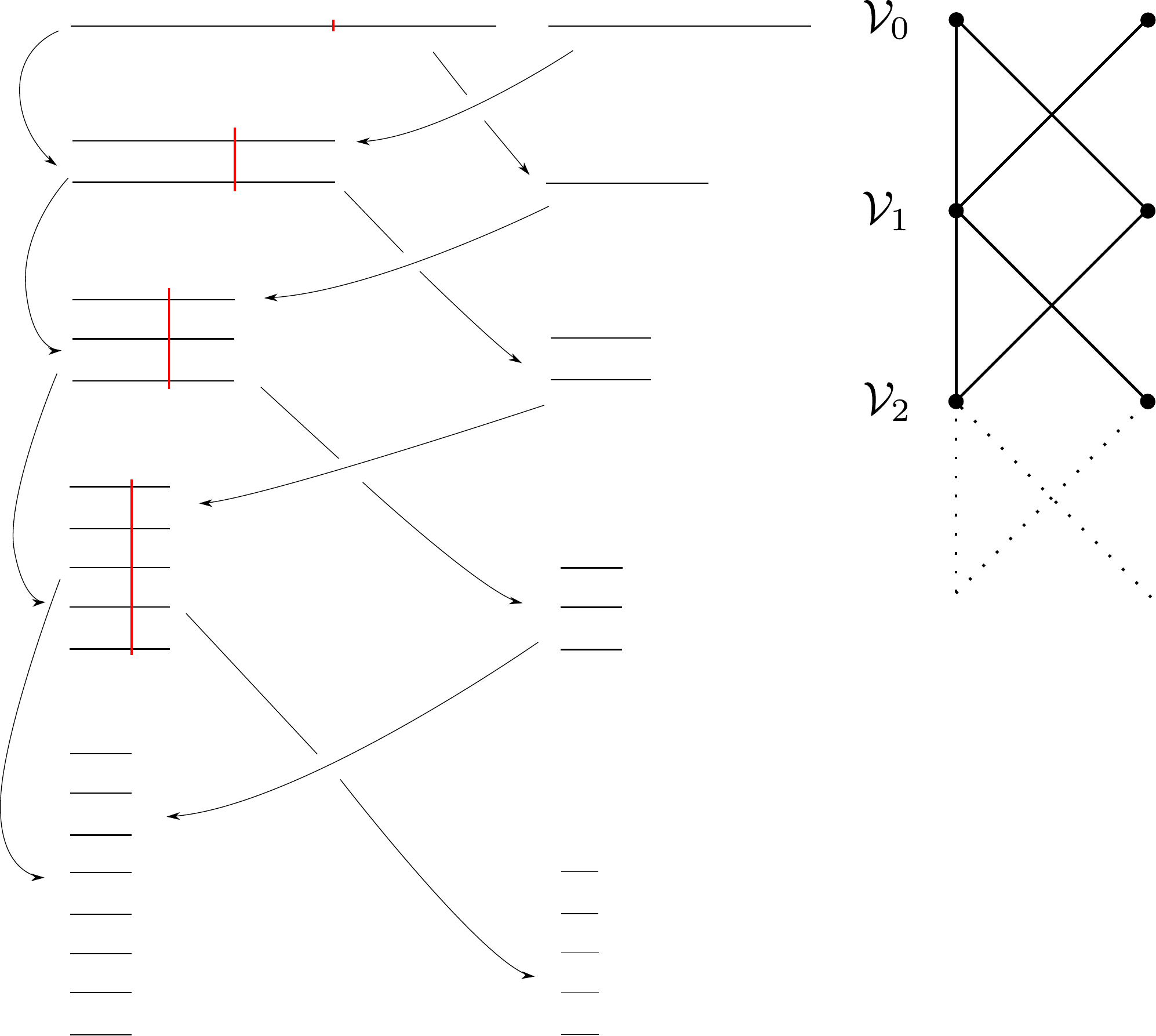}
  \caption{On the right, the positive part of the Bratteli diagram given by the matrix of the Fibonacci substitution. On the left, the cutting and stacking process given by this Bratteli diagram: the red lines represent the cutting of the stacks. At stage $k$, there is a map $T_k$ whose domain is the interior of the intervals of the stacks at stage $k$ who are not at the very top. The map $T_k$ maps the $i^{th}$ level of this stack isometrically to the one above it. The pointwise limit $T^* = \lim_{k\rightarrow \infty}$ gives an interval exchange transformation on countably infinite intervals with countably infinite points of discontinuity. The surface constructed from the identifications given by $T^*$ is depicted in Figure \ref{fig:fib}.}
  \label{fig:CAS}
\end{figure}

Starting with the positive part of $\mathcal{B}$, the cutting-and-stacking procedure is an iterative procedure which starts as follows. The $0^{th}$ stage of the procedure consists of having two intervals of length $w^+(v_1^0)$ and $w^+(v_2^0)$, respectively. To get to the first stage of this procedure, we cut the interval associated with $v_1^0$ into two intervals of widths $w^+(v_1^0)\frac{\sqrt{5}-1}{2}$ and $w^+(v_1^0)\frac{3-\sqrt{5}}{2}$, respectively. We now create two stacks which will be identified with the vertices $v_1^1$ and $v_2^1$ as follows: the first one (corresponding to $v_1^1$) consists of two intervals which represent the two edges coming into $v_1^1$. That is, this stack consists of the left subinterval of $v_1^0$ and the interval represented by $v_2^0$. The ordering $\leq_r$ mandates that in this stack the left subinterval of $v_1^0$ is placed under the interval representing $v_2^0$. The second tower, the one representing $v_2^1$, consists of the right subinterval of $v_1^0$, corresponding to the edge which goes from $v_1^0$ to $v_2^1$. We can now define a map $T_1$ whose domains is the interiors of intervals in all stacks which are not at the top. Since only the first stack so far has more than one interval, the domain of the map is the interior of the left subinterval of $v_1^0$, that is, the subinterval representing the edge going from $v_1^0$ to $v_1^1$. The map send the subinterval on the bottom of the first stack isometrically to the interval directly above it.

We now proceed iteratively: The $k^{th}$ stage of the procedure consists of having two stacks of intervals of length $w^+(v_1^k)$ and $w^+(v_2^k)$, respectively. To get to the $k+1^{st}$ stage of this procedure, we cut the stack associated with $v_1^k$ into two stacks of widths $w^+(v_1^k)\frac{\sqrt{5}-1}{2}$ and $w^+(v_1^k)\frac{3-\sqrt{5}}{2}$, respectively. We now create two stacks which will be identified with the vertices $v_1^{k+1}$ and $v_2^{k+1}$ as follows: the first one (corresponding to $v_1^{k+1}$) consists of intervals which came from the left substack of $v_1^k$ placed under the intervals coming from the stack $v_2^k$. The second stack, the one representing $v_2^{k+1}$, consists of the right substack of $v_1^k$, corresponding to the edge which goes from $v_1^k$ to $v_2^{k+1}$. We can now define a map $T_{k+1}$ whose domains is the interiors of intervals in all stacks which are not at the top. This map will coincide with $T_{k}$ wherever their domains coincide, and it will be an extension of it. The map sends any subinterval which is not at the top of the stack isometrically to the interval directly above it. This procedure gives a sequence of maps $\{T_k\}_k$ defined as local isometries on various subsets of $[0,1]$. Let $T^* = \lim_{k\rightarrow \infty} T_k$ be the pointwise limit of these maps (wherever this is defined).

Consider now $\bar{S} = [0,w^+(v_1^0)]^2\sqcup [w^+(v_1^0), 1]^2$, the disjoint union of two rectangles. The map $T_k$ will define identifications on the boundary of $\bar{S}$. Indeed, if $(x,y)\in \bar{S}$ is a point on the top edge and $x$ is not a point of discontinuity of $T^*$, we identify $(x,y)\sim (T^*(x),y')$, where $y'$ is either $0$ or $w^+(v_1^0)$, depending on whether $T^*(x)< w^+(v_1^0)$ or not. Since $\mathcal{B}$ is stationary, we also use $T^*$ to give identifications to the left/right edges of $\bar{S}$ in the same way. As such, the surface $S(\mathcal{B},w^\pm, \leq_{r,s})$ is defined to be $(\bar{S}/\sim)- \Sigma$, where $\sim$ is the identification on $\partial \bar{S}$ given by $T^*$ and explained above, and $\Sigma\subset \partial \bar{S}$ is a small subset given by points of discontinuity of $T^*$. 

\begin{figure}[t]
  \centering
  \includegraphics[width=0.9\textwidth]{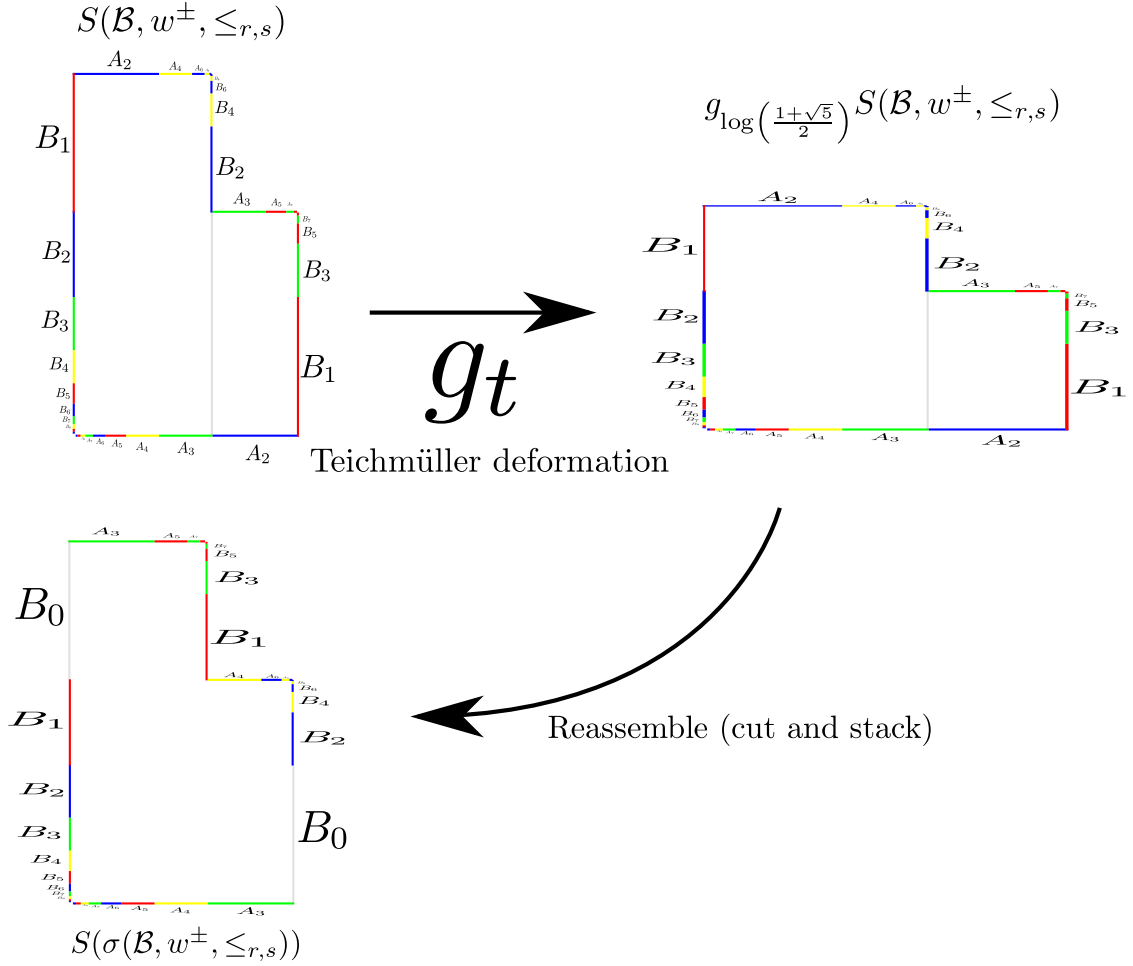}
  \caption{Depicting Proposition \ref{prop:shift}: the relationship between the Teichm\"uller deformation and the shifting operation on the Bratteli diagram starting with the surface in Figure \ref{fig:fib}.}
  \label{fig:deform}
\end{figure}

It remains to depict the relationship between the Teichm\"uller deformation of $S(\mathcal{B},w^\pm, \leq_{r,s})$ and the shift on $(\mathcal{B},\leq_{r,s})$, that is, depict the content of Proposition \ref{prop:shift}. By (\ref{eqn:rTimes}) in this case (since $\mathcal{B}$ is stationary) we get that $t_k = k\log\left(\frac{1+\sqrt{5}}{2}\right)$ for all $k\in\mathbb{N}$. This is depicted in Figure \ref{fig:deform}.

\section{$|\mathcal{V}_k| = 1$ for infinitely many $k>0$ implies strict ergodicity}
\label{sec:strict}
Here I will briefly sketch the proof of the following statement:
\begin{proposition}
Let $B = (V,E)$ be a Bratteli diagram and suppose that $|V_k| = 1$ for infinitely many $k$. Then the tail equivalence relation is minimal and supports a unique invariant probability measure.
\end{proposition}
The proof of minimality is straight forward: let $x = (x_1,x_2,\dots)\in X_B$ be an infinite path starting at $V_0$. Being minimal means that for any other $x' = (x_1',x_2',\dots) \in X_B$, for any $k>0$, we have that there is a path $y = (y_1,y_2,\dots)$ such that $x_i = y_i$ or all $i\in\{1,\dots, k\}$ and $y_i = x_i'$ for all $i$ large enough. This is guaranteed by the fact that $|V_k| = 1$ for infinitely many $k>0$. So the tail equivalence relation is minimal.

That the tail equivalence is ergodic can be shown as follows. Since $|V_k| = 1$ for infinitely many $k$, for any specific choice of orders $\leq_{r,s}$ and weight function $w$, the weighted, ordered diagram $(B,w,\leq_{r,s})$ defines defines a cutting and stacking transformation of $[0,1]$ with one tower. Moreover, by letting $(\mathcal{B},w^\pm,\leq_{r,s})$ be the weighted, ordered Bratteli diagram whose positive part is $(B,w,\leq_{r,s})$ and whose negative part is given by the matrix $\mathcal{F}_k = \mathrm{Id}_{|V_0|}$ for all $k<0$, we can construct a flat surface $S(\mathcal{B},w^\pm, \leq_{r,s})$ of area 1 as explained in \S \ref{subsec:SirfDiag}. Denote by $k_i\rightarrow \infty$ a subsequence satisfying $|V_{k_i}| = 1$ for all $i$.

Now we deform the surface $S(\mathcal{B},w^\pm, \leq_{r,s})$ using the Teichm\"uller deformation. In particular, we consider the deformed surfaces $g_{t_{k_i}}S(\mathcal{B},w^\pm, \leq_{r,s})$, where $t_{k_i}$ are the times given by (\ref{eqn:rTimes}). By Proposition \ref{prop:shift}, the surfaces $g_{t_{k_i}}S(\mathcal{B},w^\pm, \leq_{r,s})$ can be represented by a square along with identifications on its boundary. As such, all geometric quantities relevant in Theorem \ref{thm:integrability} are bounded independently of $k_i$. An estimate equal to (\ref{eqn:bigDivergence}) can be therefore derived and thus, through Theorem \ref{thm:integrability}, get ergodicity for the vertical flow on $S(\mathcal{B},w^\pm, \leq_{r,s})$, which implies ergodicity for the measure defined by the weight function $w$ for the tail equivalence on $X_B$. To get unique ergodicity, the argument used in \S \ref{subsec:upgrade} works, and so we get strict ergodicity.
\bibliographystyle{amsalpha}
\bibliography{biblio}

\end{document}